\documentclass[11pt,reqno]{amsart}
\usepackage{titlesec}
\titleformat{\section}[block]{\large \bfseries}{\arabic{section}}{1em}{}[]
\titlespacing{\section}{0pt}{*1.5}{*1.1}
\titleformat{\subsection}[block]{\normalsize \bfseries}{\arabic{section}.\arabic{subsection}}{1em}{}[]
\titlespacing{\subsection}{0.5cm}{*4}{*1.5}

\usepackage{geometry}
\usepackage{lipsum}
\usepackage{amsfonts}
\usepackage{graphicx}
\usepackage{epstopdf}
\usepackage{algorithm}

\usepackage[rotateright]{rotating}
\usepackage{subfigure,fancybox}
\usepackage{amscd,amstext}
\usepackage{amsmath}
\usepackage{amssymb}
\usepackage{color}
\usepackage{multirow}
\usepackage{mathrsfs}
\usepackage{diagbox}
\usepackage{bbm}
\numberwithin{figure}{section}
\numberwithin{table}{section}

\usepackage{enumitem}

\ifpdf
  \DeclareGraphicsExtensions{.eps,.pdf,.png,.jpg}
\else
  \DeclareGraphicsExtensions{.eps}
\fi

\theoremstyle{plain}
\newtheorem{theorem}{Theorem}[section]

\newtheorem{lemma}{Lemma}[section]

\newtheorem{remark}{Remark}[section]

\usepackage{amsopn}

\numberwithin{equation}{section}
\renewcommand{\theequation}{\thesection.\arabic{equation}}

\usepackage[pagewise]{lineno}

\begin{document}

\title[Analysis and Elimination of Numerical Pressure Dependency]{Analysis and Elimination of Numerical Pressure Dependency in Coupled Stokes-Darcy Problem}

\author{{Jiachuan Zhang}}
\thanks{School of Physical and Mathematical Sciences, Nanjing Tech University, Nanjing, Jiangsu, 211816, P. R. China. (zhangjc@njtech.edu.cn)}

\begin{abstract}
This paper analyses the classical mixed finite element method (FEM) and a pressure-robust variant with divergence-free reconstruction operators for the coupled Stokes-Darcy problem. 
Its main contribution is to provide viscosity-explicit \textit{a priori} error estimates that clearly distinguish the pressure dependence of the two discretizations: the velocity error of the classical scheme depends on both the exact pressure and the viscosity, whereas the pressure-robust method eliminates both entirely. 
Moreover, we derive pressure error estimates and quantify their dependence on the exact solution and model parameters. 
Two-dimensional numerical experiments validate the theoretical findings, including higher-order tests up to polynomial degree three and a lid-driven cavity benchmark with a piecewise linear interface. 
The implementation code is made publicly available to facilitate reproducibility.
\medskip

\noindent{\bf Keywords:}~~ pressure-robustness,Stokes-Darcy problem, reconstruction operator, error estimate

\noindent{\bf \text{Mathematics Subject Classification} :}~~65N15, 65N30, 76D07.

\end{abstract}

\maketitle

\section{Introduction}\label{sec:introduction}
This paper concerned with pressure-robust mixed FEMs of the coupled Stokes-Darcy problem. This problem is involved in energy engineering, battery technology, biomedicine, and so on. There are many numerical method solving this problem \cite{discacciatiMathematicalNumericalModels2002, Layton2002, discacciatiConvergenceAnalysisSubdomain2004,riviereLocallyConservativeCoupling2005,
gaticaConformingMixedFiniteelement2008a,correaUnifiedMixedFormulation2009,
gaticaAnalysisFullymixedFinite2011,vassilevDomainDecompositionCoupled2014,
wenDiscontinuousGalerkinMethod2020,caoExtendedFiniteElement2022}. Especially for reference \cite{Layton2002}, Layton et al. have established a complete theoretical analysis framework for a class of mixed FEMs with respect to two schemes. However, its convergence analysis of the consistency error depends on a Lagrange multiplier. 
This makes it unclear how the velocity error depends on pressure.

For pressure dependence arising from discrete divergence-free constraints, the review article by John et al. \cite{johnDivergenceConstraintMixed2017} systematically examines three methodological frameworks \cite{falkStokesComplexesConstruction2013, guzmanConformingDivergencefreeStokes2014, cockburnNoteDiscontinuousGalerkin2007, wangNewFiniteElement2007, linkeDivergencefreeVelocityReconstruction2012, Linke2016}. A notable strategy, firstly proposed in \cite{linkeDivergencefreeVelocityReconstruction2012},  involves implementing divergence-free reconstruction operators exclusively on the right-hand test functions while preserving the original variational formulation on the left-hand side. Remarkably, this relatively straightforward adaptation has demonstrated the capacity to confer pressure robustness when applied to numerous classical numerical schemes \cite{Linke2016, ledererDivergencefreeReconstructionOperators2017,linkePressurerobustnessQuasioptimalPriori2020}.

The coupled Stokes-Darcy problem involves discrete divergence-free constraints as well as interface normal-continuity conditions. 
Within the framework of divergence-free reconstruction operators, several lemmas and properties developed for the Stokes problem in~\cite{Linke2016} cannot be applied directly to the coupled Stokes-Darcy setting due to the additional interface coupling. 
To the best of our knowledge, Lv et al.~\cite{lvPressurerobustMixedFinite2024} proposed a pressure-robust method based on reconstruction operators, but their results are restricted to lowest-order discretizations. 
Similarly, Jia et al.~\cite{jiaPressurerobustWeakGalerkin2024} developed a pressure-robust weak Galerkin method, which is limited to two-dimensional cases. 
In our previous work~\cite{Zhang2025}, we extended the divergence-free reconstruction operator to the Stokes-Darcy optimal control problem and obtained a pressure-robust discretization in two and three dimensions for pressure finite elements of degree at least one. 
That analysis shows that the velocity error of the pressure-robust scheme becomes independent of the exact pressure.
However, the final velocity bound in~\cite{Zhang2025} still contains an implicit viscosity dependence (through the hidden constants) and may lead to a loss of robustness in the small-viscosity regime.
Moreover, error estimates for the pressure variable are not provided in~\cite{Zhang2025}.

These observations motivate a dedicated pressure-robust analysis for the coupled Stokes-Darcy problem within the reconstruction framework.
The goal of the present work is to derive viscosity-explicit error estimates and to obtain conclusions tailored to this coupled system.
Our main contributions are summarized as follows. First, we establish a viscosity-explicit pressure-robust error analysis for the coupled Stokes-Darcy problem.
In particular, we show that the velocity error of the classical discretization depends on both the exact pressure and the viscosity, whereas the velocity error of the pressure-robust discretization is independent of both.
This conclusion is essentially different from the optimal control setting studied in~\cite{Zhang2025}, where the final pressure-robust velocity bound still exhibits an implicit viscosity sensitivity. Second, we develop a streamlined velocity error analysis for the coupled Stokes-Darcy problem.
Let $\boldsymbol{u}$ and $\boldsymbol{u}_h$ denote the exact velocity and its finite element approximation, respectively.
Our analysis is based on the auxiliary velocity projector $S_h\boldsymbol{u}$ introduced in~\cite{Zhang2025}.
The estimate for $\boldsymbol{u}_h-S_h\boldsymbol{u}$ obtained here admits a sharper upper bound that depends only on the best-approximation error of the pressure, whereas the corresponding bound in~\cite{Zhang2025} also involves velocity-related approximation terms. Finally, we further derive pressure error estimates and validate the theoretical findings by numerical experiments, including higher-order tests (up to $k=3$) and a lid-driven cavity benchmark with a piecewise linear interface. The implementation code is made publicly available to support reproducibility.

The remainder of this paper is organized as follows. Section~2 presents the Stokes-Darcy coupled system and derives its variational formulation. Section~3 constructs the auxiliary projector and analyzes the classical mixed FEM discretization, yielding viscosity-explicit estimates that reveal the pressure-dependent velocity error mechanisms.
Section~4 proposes a pressure-robust discretization through divergence-free reconstruction operators, eliminating pressure and viscosity dependence in velocity error estimates. Section~5 validates theoretical predictions via numerical benchmarks comparing classical and pressure-robust methods under varying pressure and viscosity conditions. Finally, Section~\ref{sec:conclusion} summarizes the key findings of this study and outlines potential extensions for future researches.

\section{Problem statements and variational formulation}
Let $\Omega^i$ ($i=s,d$) denote bounded, simply connected polygonal or polyhedral in $\mathbb{R}^N$ ($N=2,3$) with interface $\Gamma=\overline{\Omega^s}\cap \overline{\Omega^d}$, and boundaries $\Gamma^i=\partial\Omega^i\backslash\Gamma$. Unit normal vector $\boldsymbol{n}^i$ are oriented outward from $\partial\Omega^i$, with $\boldsymbol{n}^s=-\boldsymbol{n}^d$ on $\Gamma$. Figure~\ref{fig:geometry} shows a schematic diagram of the geometric configuration.
This paper is concerned with pressure-robust FEMs for the Stokes-Darcy problem described by the following systems. 
The free fluid domain $\Omega^s$ and porous media domain $\Omega^d$ are governed respectively by:
\begin{align}\label{eqn:state Stokes}
&-2\mu\nabla\cdot D(\boldsymbol{u}^s)+\nabla p^s=\boldsymbol{f}^s \mbox{~in~} \Omega^s,\quad
\nabla\cdot \boldsymbol{u}^s=g^s \mbox{~in~} \Omega^s,\quad
\boldsymbol{u}^s=0 \mbox{~on~} \Gamma^s,
\end{align}
and
\begin{align}\label{eqn:state Darcy}
&\mu K^{-1}\boldsymbol{u}^d+\nabla p^d=\boldsymbol{f}^d \mbox{~in~} \Omega^d,\quad
\nabla\cdot \boldsymbol{u}^d=g^d \mbox{~in~} \Omega^d,\quad
\boldsymbol{u}^d\cdot \boldsymbol{n}^d=0 \mbox{~on~} \Gamma^d,
\end{align}
with interface conditions:
\begin{align}\label{eqn:state boundary 1}
&\boldsymbol{u}^s\cdot\boldsymbol{n}^s+\boldsymbol{u}^d\cdot\boldsymbol{n}^d=0,\\
\label{eqn:state boundary 2}
&p^s-2\mu D(\boldsymbol{u}^s)\boldsymbol{n}^s\cdot\boldsymbol{n}^s=p^d,\\
\label{eqn:state boundary 3}
&\boldsymbol{u}^s\cdot\boldsymbol{\tau}_j=-2\frac{\sqrt{\kappa_j}}{\alpha_1}D(\boldsymbol{u}^s)
\boldsymbol{n}^s\cdot\boldsymbol{\tau}_j,\quad j=1,\cdots,N-1.
\end{align}

\begin{figure}[htbp]
  \centering
  \includegraphics[width=0.6\textwidth]{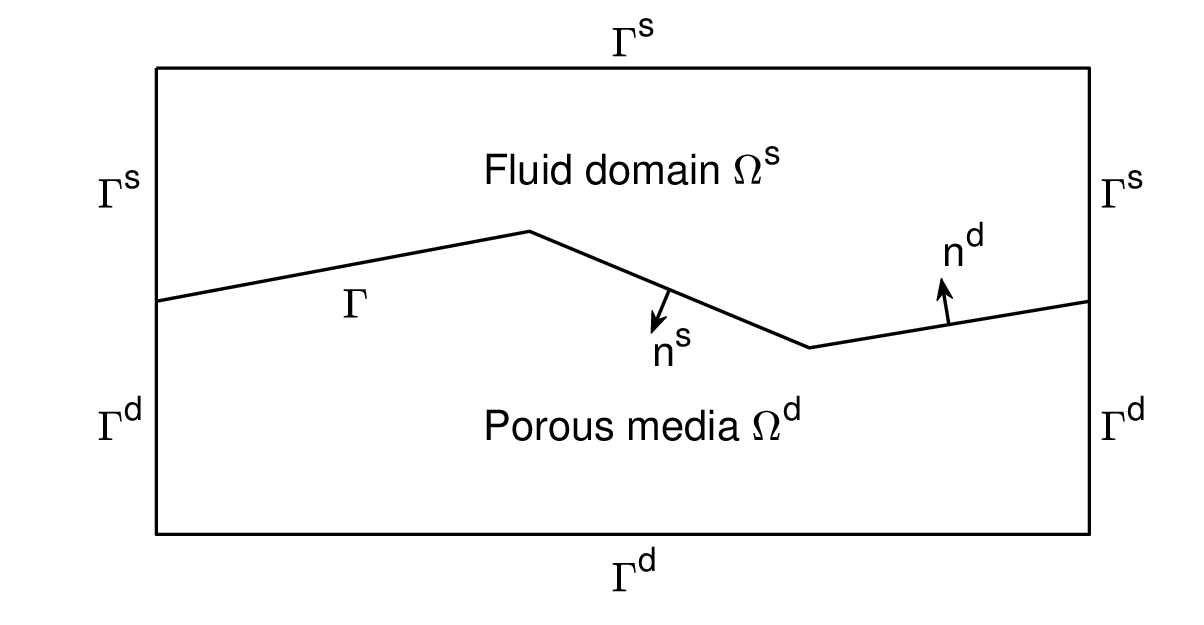}
\caption{The fluid domain and porous media domain}
\label{fig:geometry}
\end{figure}

The solution pairs $(\boldsymbol{u}^s,p^s)$ and $(\boldsymbol{u}^d,p^d)$ satisfy Stokes and Darcy equations in their respective domains, where
\begin{enumerate}[itemindent=-1.5em]
\item[$\bullet$] $\boldsymbol{f}^i$ and $g^i$ represent body forces and divergence constraints satisfying $\int_{\Omega^s}g^s+\int_{\Omega^d}g^d=0$;
\item[$\bullet$] $\mu$ denotes fluid viscosity;
\item[$\bullet$] $D(\boldsymbol{u}^s)=(\nabla\boldsymbol{u}^s+(\nabla\boldsymbol{u}^s)^T)/2$ is the strain tensor;
\item[$\bullet$] $K$ is a symmetric positive definite permeability tensor satisfying
\begin{align*}
K_L\xi^T\xi\leq \xi^TK\xi\leq K_U\xi^T\xi, \quad \forall \xi\in\mathbb{R}^N,
\end{align*} 
for some constants $0< K_L\leq K_U<\infty$;
\item[$\bullet$] $\kappa_j=\boldsymbol{\tau}_j\cdot K\cdot\boldsymbol{\tau}_j$ with $\boldsymbol{\tau}_j$ being tangent vectors to $\Gamma$;
\item[$\bullet$] $\alpha_1$ is an empirical Beavers-Joseph-Saffman (BJS) parameter.
\end{enumerate}
The interface conditions enforce: (\ref{eqn:state boundary 1}) normal continuity, (\ref{eqn:state boundary 2}) balance of normal forces, and (\ref{eqn:state boundary 3}) the BJS slip condition.

Let $H^m(E)$ denote standard Sobolev spaces with norm $\|\cdot\|_{m,E}$ and seminorm $|\cdot|_{m,E}$ for subdomain $E\subset\mathbb{R}^N$. 
In particular, $H^0(E)=L^2(E)$ and the subscript $m$ will be dropped from the norms. The norms are applicable to both vector valued functions in space $[H^m(E)]^N$ and tensor valued functions in space $[H^m(E)]^{N\times N}$ as well.
Let $(\cdot,\cdot)_E$ denote the $L^2(E),[L^2(E)]^N$, and $[L^2(E)]^{N\times N}$ inner product for scalar, vector, and tensor valued functions, respectively. 
In the above definition, if $E=\Omega^i, i=s,d$, we will abbreviate as $\|\cdot\|_{k,i}$,
$|\cdot|_{k,i}$, and $(\cdot,\cdot)_i$. We also use $\langle\cdot,\cdot\rangle_{\Gamma}$ and $\|\cdot\|_{\Gamma}$ to denote the $L^2(\Gamma)$ inner product and norm, respectively, for scalar and vector valued functions.

We define:
\begin{align*}
L_0^2(E)&=\{q\in L^2(E)~|~\int_E q=0\},\\
H(div;E)&=\{\boldsymbol{v}\in [L^2(E)]^N~|~\nabla\cdot\boldsymbol{v}\in L^2(E)\}.
\end{align*}
The velocity spaces are:
\begin{align*}
V^s &= \{\boldsymbol{v}^s \in [H^1(\Omega^s)]^N \ | \ \boldsymbol{v}^s = \boldsymbol{0} \ \text{on } \Gamma^s\}, \\
V^d &= \{\boldsymbol{v}^d \in H(div;\Omega^d) \ | \ \boldsymbol{v}^d\cdot\boldsymbol{n}^d = 0 \ \text{on } \Gamma^d\},
\end{align*}
equipped with the norm
\begin{align*}
\|\boldsymbol{v}\|_X=(|\boldsymbol{v}^s|_{1,s}^2+\|\boldsymbol{v}^d\|_d^2
+\|\nabla\cdot\boldsymbol{v}^d\|_d^2)^{1/2}.
\end{align*}
The composite velocity space incorporating interface condition (\ref{eqn:state boundary 1}) is:
\begin{align*}
V=\{\boldsymbol{v}=(\boldsymbol{v}^s,\boldsymbol{v}^d)\in V^s\times V^d~|~
\langle\boldsymbol{v}^s\cdot\boldsymbol{n}^s+\boldsymbol{v}^d\cdot\boldsymbol{n}^d,\lambda\rangle_{\Gamma}=0, \forall\lambda\in\Lambda\},
\end{align*}
where $\Lambda=H_{00}^{1/2}(\Gamma)\subset L^2(\Gamma)$ follows the definition in \cite{Layton2002}. The pressure space is $Q=L_0^2(\Omega)$ for $\Omega=\Omega^s\cup\Omega^d$.

The weak formulation of (\ref{eqn:state Stokes})-(\ref{eqn:state boundary 3}) reads: Find
$(\boldsymbol{u},p)\in V\times Q$ satisfying
\begin{align}\label{eqn:state weak 1}
a(\boldsymbol{u},\boldsymbol{v})+b(\boldsymbol{v},p)&=(\boldsymbol{f},
\boldsymbol{v}),
\quad \forall\boldsymbol{v}\in V,\\
\label{eqn:state weak 2}
b(\boldsymbol{u},q)&=(g,q), \quad \forall q\in Q,
\end{align}
with bilinear forms
\begin{align*}
&a(\boldsymbol{u},\boldsymbol{v})\\
=&a_s(\boldsymbol{u}^s,\boldsymbol{v}^s)
+a_d(\boldsymbol{u}^d,\boldsymbol{v}^d)
+a_I(\boldsymbol{u}^s,\boldsymbol{v}^s)\\
=&2\mu(D(\boldsymbol{u}^s),D(\boldsymbol{v}^s))_s+\mu(K^{-1}\boldsymbol{u}^d,\boldsymbol{v}^d)_d
+\sum_{j=1}^{N-1}\frac{\alpha_1\mu}{\sqrt{\kappa_j}}\langle\boldsymbol{u}^s\cdot\boldsymbol{\tau}_j,
\boldsymbol{v}^s\cdot\boldsymbol{\tau}_j\rangle_{\Gamma},\quad \forall \boldsymbol{u},\boldsymbol{v}\in V^s\times V^d,\\
&b(\boldsymbol{v},q)=-(\nabla\cdot\boldsymbol{v}^s,q)_s-(\nabla\cdot\boldsymbol{v}^d,q)_d,
\quad \forall \boldsymbol{v}\in V^s\times V^d, q\in Q.
\end{align*}

The normal continuity condition (\ref{eqn:state boundary 1}) is enforced through the space $V$, while conditions (\ref{eqn:state boundary 2}) and (\ref{eqn:state boundary 3}) emerge naturally in the weak formulation. The well-posedness of this system is established in Theorem~3.1 of \cite{Layton2002}.

\section{Classical FEM discretization}

Let $\mathcal{T}_h$ be a conforming and shape-regular triangulation, which consists of simplices and matches at $\Gamma$ \cite{Layton2002}. The set of edges or faces in $\mathcal{T}_h$ is denoted by $\mathcal{E}_h$. Moreover, we define three useful sub-triangulation as following
\begin{align*}
\mathcal{E}_h(\Gamma)=\{e\in \mathcal{E}_h~|~ e\subset\Gamma\},
\quad \mathcal{T}_h(\Omega^i)=\{T\in \mathcal{T}_h~|~ T\subset\Omega^i\}, \quad i=s,d.
\end{align*}
For $T\in\mathcal{T}_h$, let $h_T$ denote the diameter of the polygon or polyhedra $T$ and $h=\max_{T\in\mathcal{T}_h} h_T$. For the discretization of the fluid's variables in $\Omega^s$, we choose finite element space $V_h^s\subset V^s$, $Q_h^s\subset L^2(\Omega^s)$ introduced in Section 4.1.1 in \cite{Linke2016}, which are LBB-stable
This velocity space is achieved by enriching the space $P_k(T),k\geq 2$, of continuous, piece-wise polynomial functions of degree less than or equal to $k$ with suitable bubble functions.
The pressure space in $\Omega^s$ contains the functions with degree $k-1$ in $\Omega^s$, i.e. $Q_h^s=\{q_h\in L^2(\Omega^s)~|~ q_{h|T}\in P_{k-1}(T), T\in \mathcal{T}_h(\Omega^s)\}$. 

For any element $T\in\mathcal{T}_h$, let $\{z_0,\cdots,z_d\}$ denote its vertices. We define barycentric coordinates $\lambda_j\in P_1(T), 0\leq j\leq N$ uniquely such that $\lambda_j(z_{l})=\delta_{jl}$, where $\delta_{jl}$ denotes the Kronecker delta function. Let $b_T$ be the product of barycentric coordinates related to the element $T$, and let $\tilde{P}_l(T)$ denote the space of homogeneous polynomials of degree $l$. 
Indeed, we choose
\begin{align*}
V_h^s=\{\boldsymbol{v}_h^s\in V^s~|~\boldsymbol{v}_{h|T}^s\in \boldsymbol{P}_k^+(T), T\in\mathcal{T}_h(\Omega^s)\}\quad\mbox{with}\quad
\boldsymbol{P}_k^+(T)=[P_k(T)\oplus span\{b_T\tilde{P}_{k-2}(T)\}]^2,
\end{align*}
in two-dimensional case and 
\begin{align*}
V_h^s=\{\boldsymbol{v}_h^s\in V^s~|~\boldsymbol{v}_{h|T}^s\in \boldsymbol{P}_k^+(T)], T\in\mathcal{T}_h(\Omega^s)\},
\end{align*}
for three-dimensional case with
\begin{align*}
\boldsymbol{P}_k^+(T)&=[P_k(T)\oplus span\{b_T(\tilde{P}_{k-2}(T)\oplus\tilde{P}_{k-3}(T))\}]^3,
\quad k\geq 3,\\
\boldsymbol{P}_k^+(T)&=[P_k(T)\oplus b_T(\tilde{P}_{k-2}(T))]^3\oplus span\{\boldsymbol{p}_1,\boldsymbol{p}_2,\boldsymbol{p}_3,\boldsymbol{p}_4\},\quad k=2
\end{align*}
with the face bubble functions 
$\boldsymbol{p}_j=\boldsymbol{n}_j\prod_{l=1}^N\lambda_{j_l}$,
where $j_1,\cdots,j_N$ are different indexes from the set $\{1,\cdots,N+1\}/\{j\}$ and $\boldsymbol{n}_j$ is the outer normal of the face opposite the vertex $\lambda_j=1$.

The porous medium problem in $\Omega^d$, 
we choose mixed finite element spaces to be the $RT_{k-1}$ spaces \cite{Raviart1977} for the normal velocity vanishing on boundary $\Gamma^d$, and the pressure space similar to fluid domain, denoted by $V_h^d$  and $Q_h^d$. It is known for these choices that $V_h^d\subset V^d, Q_h^d\in L^2(\Omega^d)$, and
$\nabla\cdot V_h^d=Q_h^d$.

On the interface $\Gamma$, define another finite element space, which is the normal trace of $V_h^d$ on $\Gamma_h$, as following
\begin{align*}
\Lambda_h=\{\lambda_h\in L^2(\Gamma)~|~\mu_{h|e} \in [P_{k}(e)]^N, \forall e\in\mathcal{E}_h(\Gamma)\}.
\end{align*}

With these spaces $V_h^s, V_h^d$, and $\Lambda_h$, define
\begin{align}\label{eqn:definition of Vh}
V_h=\{\boldsymbol{v}_h=(\boldsymbol{v}_h^s,\boldsymbol{v}_h^d)\in V_h^s\times V_h^d~|~
\langle\boldsymbol{v}_h^s\cdot\boldsymbol{n}^s
+\boldsymbol{v}_h^d\cdot\boldsymbol{n}^d,\lambda_h\rangle_{\Gamma}=0,\forall \lambda_h\in\Lambda_h\}.
\end{align}
Note that, since function $\lambda_h\in\Lambda_h$ does not in general vanish on $\partial\Gamma$ ($\Lambda_h\not\subset\Lambda$), the space $V_h$ is nonconforming ($V_h\not\subset V$).
With spaces $Q_h^s$ and $Q_h^d$, define $Q_h=\{q\in Q_h^s\times Q_h^d~|~(q,1)_{\Omega}=0\}$, which is vanishing integral in $\Omega$ and satisfy $Q_h\subset Q$.

The classical variational discretization of (\ref{eqn:state weak 1})-(\ref{eqn:state weak 2}) solves the following discrete problem: seek $(\boldsymbol{u}_h,p_h)\in V_h\times Q_h$ such that
\begin{align}
\label{eqn:PDE problem clssical}
\begin{aligned}
a(\boldsymbol{u}_h,\boldsymbol{\psi}_h)+b(\boldsymbol{\psi}_h,p_h)
&=(\boldsymbol{f},\boldsymbol{\psi}_h),
 &&\forall \boldsymbol{\psi}_h\in V_h,\\
b(\boldsymbol{u}_h,\phi_h)&=(g,\phi_h), &&\forall \phi_h\in Q_h.
\end{aligned} 
\end{align}

The inf-sup condition of the pair $V_h^s\times Q_h^s$ can be get from Lemma~4.1 in \cite{Layton2002} combining with the following Lemma~\ref{lem:projection operatro Upslion_1}. Then, the inf-sup conditions for the coupled problem holds form Lemma~4.3 in \cite{Layton2002} which determines the solvability of the discrete problem. Although the pressure space $Q_h^s$ is not restricted to have zero mean over $\Omega^s$, i.e., $Q_h^s\subset L^2(\Omega^s)$, not $L_0^2(\Omega^s)$, the proof of Lemma~\ref{lem:projection operatro Upslion_1}
is similar to the usual one verified in the literature (VI.4 of \cite{Brezzi1991}). 
To maintain the integrity of the text, we still provide the proof in Appendix.

\begin{remark}
In this work, we use $\mathbbm{a}\lesssim \mathbbm{b}$ when there exists a constant $c$ independent of $\mathbbm{a},\mathbbm{b},h,\mu$ such that $\mathbbm{a}\leq c\mathbbm{b}$.
\end{remark} 

\begin{lemma}\label{lem:projection operatro Upslion_1}
There exists an operator $\Upsilon_h^s: V^s\rightarrow V_1^s\subset V_h^s$ satisfying, for any $T\in \mathcal{T}_h(\Omega^s)$ and all $\boldsymbol{v}^s\in V^s$,
\begin{align}
\int_T\nabla\cdot(\Upsilon_h^s\boldsymbol{v}^s-\boldsymbol{v}^s)=0, \quad \mbox{~and~}\quad 
\|\Upsilon_h^s\boldsymbol{v}^s\|_s\lesssim \|\boldsymbol{v}^s\|_{1,s},
\end{align}
where $V_1^s=\{\boldsymbol{v}_h^s\in V^s~|~\boldsymbol{v}_{h|T}^s\in \boldsymbol{P}_1^+(T), T\in\mathcal{T}_h(\Omega^s)\}$ with
$\boldsymbol{P}_1^+(T)=[P_1(T)]^N\oplus span\{\boldsymbol{p}_1,\cdots,\boldsymbol{p}_{N+1}\}$. 
\end{lemma}
\begin{proof}
The proof can be found in Appendix.
\end{proof}

Define
\begin{align*}
V(g)&=\{\boldsymbol{\psi}\in V~|~b(\boldsymbol{\psi},\phi)=(g,\phi),~\forall \phi\in Q\},\\
V_h(g)&=\{\boldsymbol{\psi}_h\in V_h~|~b(\boldsymbol{\psi}_h,\phi_h)=(g,\phi_h),~\forall \phi_h\in Q_h\}.
\end{align*}
for the inf-sup stable pair $V\times Q$ and $V_h\times Q_h$ with $V_h(g)\not\subset V(g)$ 
(since the divergence-free and normal continuity across $\Gamma$ hold discretely).

From Lemma~4.4 in \cite{Layton2002} or Proposition~5.5.6 in \cite{Boffi2013}, a proved result can be shown as
\begin{align}\label{eqn:u-uh pre-estimate}
&\|\boldsymbol{u}-\boldsymbol{u}_h\|_X+\|p-p_h\| \lesssim \inf_{\boldsymbol{\psi}_h\in V_h} \|\boldsymbol{u}-\boldsymbol{\psi}_h\|_X+\frac{1}{\mu} \inf_{\phi_h\in Q_h}\|p-\phi_h\|+\frac{1}{\mu}\mathcal{H},
\end{align} 
where the third term in the right hand is the consistency error defined as
\begin{align*}
\mathcal{H}=\sup_{\boldsymbol{\psi}_h\in V_h}\frac{|a(\boldsymbol{u},\boldsymbol{\psi}_h)+b(\boldsymbol{\psi_h},p)
-(\boldsymbol{f},\boldsymbol{\psi_h})|}{\|\boldsymbol{\psi}_h\|_X}.
\end{align*}
This estimate is conducted by coupling velocity and pressure, and the consistency error $\mathcal{H}$ in \cite{Layton2002} depends on a Lagrange multiplier but does not exhibit a correlation with pressure. These make it unclear how the velocity error depends on pressure.

The analysis in this paper is based on a pressure-related functional on $V_h$, which consists of relaxed divergence-free constraint and relaxed normal continuity across $\Gamma$. The functional can defined by
\begin{align}
\label{eqn:consistency error functional}
\vartheta_p(\boldsymbol{\psi}_h)=b(\boldsymbol{\psi}_h,p)
-\langle\boldsymbol{\psi}_h^s\cdot\boldsymbol{n}^s
+\boldsymbol{\psi}_h^d\cdot\boldsymbol{n}^d,p^d\rangle_{\Gamma}
, \quad \forall \boldsymbol{\psi}_h\in V_h,
\end{align}
related to $p^d\in H^1(\Omega^d)$.
For the function exactly satisfying divergence constraint and interface normal continuity, i.e. $\boldsymbol{\psi}_h\in V(0)$, or $p\in Q_h$ and $p^d\in\Lambda_h$, this functional is always zero. For the other cases, the functional $\vartheta_p(\boldsymbol{\psi}_h)\neq0$, which leads to the discrete format being pressure-dependent. The following lemma is a direct corollary of the Lemma~3.1 in \cite{Zhang2025}.
\begin{lemma}
\label{lem:consistency error}
When $p^d\in H^1(\Omega^d)$, it can be obtained that
\begin{align*}
\vartheta_{p}(\boldsymbol{\psi}_h)
\lesssim &
\left(\inf_{\phi_h\in Q_h^s}\|p^s-\phi_h\|_s^2+
h\inf_{\phi_h\in\Lambda_h}\|p^d-\phi_h\|_{\Gamma}^2\right)^{1/2}\|\boldsymbol{\psi}_h\|_X,\quad \forall \boldsymbol{\psi}_h\in V_h(0).
\end{align*}
\end{lemma}

Next, we estimate the distance of $\boldsymbol{u}_h$ to a projector,
denoted by $S_h\boldsymbol{u}\in V_h$, that is defined by
\begin{align}\label{eqn:Sh definition}
\begin{aligned}
a(S_h\boldsymbol{u},\boldsymbol{\psi}_h)+b(\boldsymbol{\psi}_h,\delta_h)&= a(\boldsymbol{u},\boldsymbol{\psi}_h),\quad \forall\boldsymbol{\psi}_h\in V_h,\\
b(S_h\boldsymbol{u},\phi_h)&=(g,\phi_h),\quad \forall \phi_h\in Q_h.
\end{aligned}
\end{align}
Here $\delta_h$ denotes the pressure. 
Then for any $\boldsymbol{v}_h\in V_h$, from (\ref{eqn:Sh definition}) and noting $(g,\phi_h)=b(\boldsymbol{u},\phi_h)$, we have 
\begin{align*}
a(S_h\boldsymbol{u}-\boldsymbol{v}_h,\boldsymbol{\psi}_h)+b(\boldsymbol{\psi}_h, \delta_h)&=a(\boldsymbol{u}-\boldsymbol{v}_h,\boldsymbol{\psi}_h),\quad \forall\boldsymbol{\psi}_h\in V_h,\\
b(S_h\boldsymbol{u}-\boldsymbol{v}_h,\phi_h)&=b(\boldsymbol{u}-\boldsymbol{v}_h,\phi_h),\quad \forall \phi_h\in Q_h.
\end{align*}
Since $(V_h,Q_h)$ is LBB-stable shown in Lemma~4.3 in \cite{Layton2002} combing Lemma~\ref{lem:projection operatro Upslion_1}, and operator $a(\cdot,\cdot)$ is coercive for $V^s\times\{\boldsymbol{\psi}^d\in V^d~|~\nabla\cdot\boldsymbol{\psi}^d=0\}$ shown in Lemma~3.1 in \cite{Layton2002}, then the abstract bound estimates (\cite{Brezzi1991}, Chap II. Theorem~1.2) implies
\begin{align*}
\|S_h\boldsymbol{u}-\boldsymbol{v}_h\|_X\lesssim \|\boldsymbol{u}-\boldsymbol{v}_h\|_X,
\end{align*}
which equals 
\begin{align}\label{eqn:error u and Sh u}
\|\boldsymbol{u}-S_h\boldsymbol{u}\|_X\leq \inf_{\boldsymbol{v}_h\in V_h}(\|\boldsymbol{u}-\boldsymbol{v}_h\|_X+\|S_h\boldsymbol{u}-\boldsymbol{v}_h\|_X)
\lesssim \inf_{\boldsymbol{v}_h\in V_h}\|\boldsymbol{u}-\boldsymbol{v}_h\|_X.
\end{align} 
This shows convergence rates corresponding to the regularity of $\boldsymbol{u}$  and the polynomial order of $V_h$. From the triangle inequality, this approximation result 
$\|\boldsymbol{u}-\boldsymbol{u}_h\|_X^2$ is only perturbed by the term $\|S_h\boldsymbol{u}-\boldsymbol{u}_h\|_X^2$ which therefore is the primal object of interest in the a priori error analysis below.
The error estimates involving the previously defined projector $S_h:V(g)\rightarrow V_h(g)$ in (\ref{eqn:Sh definition}) needs to be discussed in the following lemma.

\begin{lemma}\label{lem:H estimate}
If the exact solutions of (\ref{eqn:state weak 1})-(\ref{eqn:state weak 2}) have the regularity
$\boldsymbol{u}^s\in [H^2(\Omega^s)]^N$, and $p^i\in H^1(\Omega^i), i=s,d$,
it holds
\begin{align}\label{eqn:error estimate pressure depend}
\begin{aligned}
\|S_h\boldsymbol{u}-\boldsymbol{u}_h\|_X^2
\lesssim &\frac{1}{\mu^2}\left(\inf_{\phi_h\in Q_h^s}\|p^s-\phi_h\|_s^2
+h\inf_{\phi_h\in\Lambda_h}\|p^d-\phi_h\|_{\Gamma}^2\right),
\end{aligned}
\end{align}
for the solution of $\boldsymbol{u}_h$ of (\ref{eqn:PDE problem clssical}) and 
the discrete approximation $S_h\boldsymbol{u}$ of 
the exact solutions $\boldsymbol{u}$ of (\ref{eqn:state weak 1})-(\ref{eqn:state weak 2}).
\end{lemma}
\begin{proof}
From (\ref{eqn:PDE problem clssical}) and (\ref{eqn:Sh definition}), it is easy to get
\begin{align}\label{eqn:error eqn uh-Shu classical}
a(\boldsymbol{u}_h-S_h\boldsymbol{u},\boldsymbol{\psi}_h)
+b(\boldsymbol{\psi}_h,p_h-\delta_h)
&=\mathcal{G}(\boldsymbol{\psi}_h),\quad \boldsymbol{\psi}_h\in V_h,
\end{align}
where
\begin{align*}
\mathcal{G}(\boldsymbol{\psi}_h)
=(\boldsymbol{f},\boldsymbol{\psi}_h)
-a(\boldsymbol{u},\boldsymbol{\psi}_h),\quad \boldsymbol{\psi}_h\in V_h.
\end{align*}
For any $\boldsymbol{\psi}_h\in V_h$, according to the definition of $a(\cdot,\cdot)$ in (\ref{eqn:state weak 1}), we have
\begin{align}
\label{eqn:u and psi_h}
a(\boldsymbol{u},\boldsymbol{\psi}_h)
&=2\mu(D(\boldsymbol{u}^s),D(\boldsymbol{\psi}_h^s))_s
+\mu(K^{-1}\boldsymbol{u}^d,\boldsymbol{\psi}_h^d)_d
+\sum_{j=1}^{N-1}\frac{\alpha_1\mu}{\sqrt{\kappa_j}}\langle\boldsymbol{u}^s\cdot\boldsymbol{\tau}_j
,\boldsymbol{\psi}_h^s\cdot\boldsymbol{\tau}_j\rangle_{\Gamma}.
\end{align}
Since $\boldsymbol{u}^s\in [H^2(\Omega^s)]^N$ and $p^i\in H^1(\Omega^i), i=s,d$, the solutions $(\boldsymbol{u},p)$ of (\ref{eqn:state weak 1})$\sim$(\ref{eqn:state weak 2}) satisfy (\ref{eqn:state Stokes})$\sim$(\ref{eqn:state boundary 3}).
Then, based on the usual integration by parts and the boundary condition (\ref{eqn:state boundary 3}), the first and the third summands in (\ref{eqn:u and psi_h}) can be derived into
\begin{align}
\label{eqn:the first summand}
&2\mu(D(\boldsymbol{u}^s),D(\boldsymbol{\psi}_h^s))_s=-2\mu(\nabla\cdot D(\boldsymbol{u}^s),\boldsymbol{\psi}_h^s)_s
+2\mu\langle D(\boldsymbol{u}^s)\boldsymbol{n}^s,\boldsymbol{\psi}_h^s\rangle_{\Gamma},\\
\label{eqn:the third summand}
&\sum_{j=1}^{N-1}\frac{\alpha_1\mu}{\sqrt{\kappa_j}}\langle\boldsymbol{u}^s\cdot\boldsymbol{\tau}_j
,\boldsymbol{\psi}_h^s\cdot\boldsymbol{\tau}_j\rangle_{\Gamma}
=\sum_{j=1}^{N-1}-2\mu\langle D(\boldsymbol{u}^s)\boldsymbol{n}^s\cdot\boldsymbol{\tau}_j
,\boldsymbol{\psi}_h^s\cdot\boldsymbol{\tau}_j\rangle_{\Gamma}.
\end{align}

From (\ref{eqn:u and psi_h}), (\ref{eqn:the first summand}), (\ref{eqn:the third summand}), and noting the Stokes-Darcy equations in (\ref{eqn:state Stokes}) and (\ref{eqn:state Darcy}), and basing on the equality
\begin{align}
\label{eqn:the fourth summand}
2\mu\langle D(\boldsymbol{u}^s)\boldsymbol{n}^s,\boldsymbol{\psi}_h^s\rangle_{\Gamma}
=\sum_{j=1}^{N-1}2\mu\langle D(\boldsymbol{u}^s)\boldsymbol{n}^s\cdot\boldsymbol{\tau}_j
,\boldsymbol{\psi}_h^s\cdot\boldsymbol{\tau}_j\rangle_{\Gamma}
+2\mu\langle D(\boldsymbol{u}^s)\boldsymbol{n}^s\cdot\boldsymbol{n}^s
,\boldsymbol{\psi}_h^s\cdot\boldsymbol{n}^s\rangle_{\Gamma},
\end{align}
we have
\begin{align*}
a(\boldsymbol{u},\boldsymbol{\psi}_h)=(\boldsymbol{f},\boldsymbol{\psi}_h)
-(\nabla p^s,\boldsymbol{\psi}_h^s)_s
-(\nabla p^d,\boldsymbol{\psi}_h^d)_d
+2\mu\langle D(\boldsymbol{u}^s)\boldsymbol{n}^s\cdot\boldsymbol{n}^s
,\boldsymbol{\psi}_h^s\cdot\boldsymbol{n}^s\rangle_{\Gamma}.
\end{align*}
By using the integration by parts for the second and third terms in the above equation, and nothing the boundary condition in (\ref{eqn:state boundary 2}), it can be obtained that
\begin{align*}
-(\nabla p^s,\boldsymbol{\psi}_h^s)_s
-(\nabla p^d,\boldsymbol{\psi}_h^d)_d
=&(p^s,\nabla\cdot\boldsymbol{\psi}_h^s)_s
-\langle p^s,\boldsymbol{\psi}_h^s\cdot\boldsymbol{n}^s\rangle_{\Gamma}
+(p^d,\nabla\cdot\boldsymbol{\psi}_h^d)_d
-\langle p^d,\boldsymbol{\psi}_h^d\cdot\boldsymbol{n}^d\rangle_{\Gamma}\\
=&-b(\boldsymbol{\psi}_h,p)-2\mu\langle D(\boldsymbol{u}^s)\boldsymbol{n}^s\cdot\boldsymbol{n}^s
,\boldsymbol{\psi}_h^s\cdot\boldsymbol{n}^s\rangle_{\Gamma}
-\langle p^d,\boldsymbol{\psi}_h^s\cdot\boldsymbol{n}^s
+\boldsymbol{\psi}_h^d\cdot\boldsymbol{n}^d\rangle_{\Gamma}
\end{align*}
Thus, combining with (\ref{eqn:consistency error functional}), we have
\begin{align}
\label{eqn:u and phi_h simple form}
\begin{aligned}
\mathcal{G}(\boldsymbol{\psi}_h)
=
b(\boldsymbol{\psi}_h,p)+\langle p^d,\boldsymbol{\psi}_h^s\cdot\boldsymbol{n}^s
+\boldsymbol{\psi}_h^d\cdot\boldsymbol{n}^d\rangle_{\Gamma}
=\vartheta_p(\boldsymbol{\psi}_h),
\end{aligned}
\end{align}
which implies the conclusion with Lemma~\ref{lem:consistency error}.

Set $\boldsymbol{\psi}_h=\boldsymbol{u}_h-S_h\boldsymbol{u}\in V_h(0)$ in (\ref{eqn:error eqn uh-Shu classical}), it can be obtained $b(\boldsymbol{u}_h-S_h\boldsymbol{u},p_h-\delta_h)=0$. Then, (\ref{eqn:error eqn uh-Shu classical}), (\ref{eqn:u and phi_h simple form}) and Lemma~\ref{lem:consistency error} imply the conclusion. 
\end{proof}

\begin{theorem}\label{the:classical estimate}
If the exact solutions of (\ref{eqn:state weak 1})-(\ref{eqn:state weak 2}) have the regularity
$(\boldsymbol{u}^s,\boldsymbol{u}^d)\in [H^r(\Omega^s)]^N\times [H^{r-1}(\Omega^d)]^N$, $\nabla\cdot \boldsymbol{u}^d\in H^{r-1}(\Omega^d)$, and $p^i\in H^{r-1}(\Omega^i), i=s,d$ for any $k\leq r\leq k+1$ with $k\geq 2$,
it holds
\begin{align}\label{eqn:error estimate regular pressure depend}
\begin{aligned}
\|\boldsymbol{u}-\boldsymbol{u}_h\|_X^2
\lesssim h^{2(r-1)}&\left(\|\boldsymbol{u}^s\|_{r,s}^2+\|\boldsymbol{u}^d\|_{r-1,d}^2
+\|\nabla\cdot\boldsymbol{u}^d\|_{r-1,d}^2\right.\\
&\left.+\frac{1}{\mu^2}\|p^s\|_{r-1,s}^2+\frac{1}{\mu^2}\|p^d\|_{r-1,d}^2\right),
\end{aligned}
\end{align}
and
\begin{align}\label{eqn:error estimate of p regular pressure depend}
\|p-p_h\|^2
\lesssim &h^{2(r-1)}(\mu^2\|\boldsymbol{u}^s\|_{r,s}^2+\mu^2\|\boldsymbol{u}^d\|_{r-1,d}^2+\mu^2\|\nabla\cdot\boldsymbol{u}^d\|_{r-1,d}^2
+\|p^s\|_{r-1,s}^2+\|p^d\|_{r-1,d}^2).
\end{align}
\end{theorem}
\begin{proof}

From (\ref{eqn:error u and Sh u}) and Lemma~\ref{lem:H estimate}, we have
\begin{align}\label{eqn:u estimate}
\|\boldsymbol{u}-\boldsymbol{u}_h\|_X^2\lesssim \inf_{\boldsymbol{\psi}_h\in V_h}\|\boldsymbol{u}-\boldsymbol{\psi}_h\|_X^2+\frac{1}{\mu^2}\left(\inf_{\phi_h\in Q_h^s}\|p^s-\phi_h\|_s^2
+h\inf_{\phi_h\in\Lambda_h}\|p^d-\phi_h\|_{\Gamma}^2\right).
\end{align}
The first term on the right-hand side of the (\ref{eqn:u estimate}) can be directly estimated using the interpolation inequality introduced in \cite{Layton2002} (in which, refer to (4.3), (4.13), (4.14), and (4.35)), i.e.
\begin{align}\label{eqn:approximation error u and z}
\begin{aligned}
&\inf_{\boldsymbol{\psi}_h\in V_h}\|\boldsymbol{u}-\boldsymbol{\psi}_h\|_X^2
\lesssim  h^{2(r-1)}(\|\boldsymbol{u}^s\|_{r,s}^2+\|\boldsymbol{u}^d\|_{r-1,d}^2+\|\nabla\cdot \boldsymbol{u}^d\|_{r-1,d}^2).
\end{aligned}
\end{align}
We now focus on estimating the second term.
If $p^i\in H^{r-1}(\Omega^i), i=s,d$ for any $k\leq r\leq k+1$ with $k\geq 2$, let $\phi_h^d$ to be the local $L^2$ projection of $p^d$ into $\{q_h\in L^2(\Omega^d)~|~q_{h|T}\in P_{k-1}(T),~T\in \mathcal{T}_h(\Omega^d)\}$.
From 
projection property and trace inequality, we have
\begin{align}\label{eqn:cosistency error p}
\begin{aligned}
&\inf_{\phi_h\in Q_h^s}\|p^s-\phi_h\|_s^2
+h\inf_{\phi_h\in\Lambda_h}\|p^d-\phi_h\|_{\Gamma}^2\\
\lesssim & h^{2(r-1)}\|p^s\|_{r-1,s}^2+ h\|p^d-\phi_h^d\|_{\Gamma}^2\\
\lesssim & h^{2(r-1)}\|p^s\|_{r-1,s}^2+ h\sum_{T\in\mathcal{T}_h(\Gamma^d)}h_T^{-1}\|p^d-\phi_h^d\|_T^2+h_T\|\nabla(p^d-\phi_h^d)\|_T^2\\
\lesssim & h^{2(r-1)}\|p^s\|_{r-1,s}^2+h\sum_{T\in\mathcal{T}_h(\Gamma^d)}h_T^{2(r-1)-1}\|p^d\|_{r-1,T}^2\\
\lesssim & h^{2(r-1)}(\|p^s\|_{r-1,s}^2+\|p^d\|_{r-1,d}^2),
\end{aligned}
\end{align}
where $\mathcal{T}_h(\Gamma^d)=\{T\in\mathcal{T}_h(\Omega^d)~|~T\cap\Gamma\neq\emptyset\}$. In the above estimates, the first and third inequalities are derived using the approximation property of the projection, while the second inequality employs the trace inequality. Thus, (\ref{eqn:u estimate}), (\ref{eqn:approximation error u and z}) and (\ref{eqn:cosistency error p}) yield (\ref{eqn:error estimate regular pressure depend}).

Next, we show the proof of (\ref{eqn:error estimate of p regular pressure depend}) with the regularity
$(\boldsymbol{u}^s,\boldsymbol{u}^d)\in [H^r(\Omega^s)]^N\times [H^{r-1}(\Omega^d)]^N$, $\nabla\cdot \boldsymbol{u}^d\in H^{r-1}(\Omega^d)$, and $p^i\in H^{r-1}(\Omega^i), i=s,d$ for any $k\leq r\leq k+1$ with $k\geq 2$. Let $P_h p$ to be the local $L^2$ projection of $p$ into $Q_h$.
From (\ref{eqn:PDE problem clssical}) and (\ref{eqn:u and phi_h simple form}), we can get
\begin{align}
\label{eqn:b operator equality}
\begin{aligned}
b(\boldsymbol{\psi}_h,P_h p-p_h)
=&b(\boldsymbol{\psi}_h,P_h p)+a(\boldsymbol{u}_h,\boldsymbol{\psi}_h)-(\boldsymbol{f},\boldsymbol{\psi}_h)\\
=&a(\boldsymbol{u},\boldsymbol{\psi}_h)+b(\boldsymbol{\psi}_h,P_h p)
-(\boldsymbol{f},\boldsymbol{\psi}_h)-a(\boldsymbol{u}-\boldsymbol{u}_h,\boldsymbol{\psi}_h)\\
=& b(\boldsymbol{\psi}_h,P_h p-p)
-\langle p^d,\boldsymbol{\psi}_h^s\cdot\boldsymbol{n}^s
+\boldsymbol{\psi}_h^d\cdot\boldsymbol{n}^d\rangle_{\Gamma}
-a(\boldsymbol{u}-\boldsymbol{u}_h,\boldsymbol{\psi}_h).
\end{aligned}
\end{align}
The first equality in the above equation employs (\ref{eqn:PDE problem clssical}), and the third utilizes (\ref{eqn:u and phi_h simple form}).

Combining the boundedness of operators $a(\cdot,\cdot)$ and $b(\cdot,\cdot)$, (\ref{eqn:b operator equality}), and 
Lemma~\ref{lem:consistency error},
we conclude that
\begin{align*}
b(\boldsymbol{\psi}_h,P_h p-p_h)\lesssim (\|P_h p-p\|+h^{1/2}\inf_{\phi_h\in\Lambda_h}\|p^d-\phi_h\|_{\Gamma}+\mu\|\boldsymbol{u}-\boldsymbol{u}_h\|_X)\|\boldsymbol{\psi}\|_X,
\end{align*}
and
\begin{align}
\label{eqn:Ph p-p}
\begin{aligned}
\|P_h p-p_h\|\lesssim &\sup_{\boldsymbol{\psi}_h\in V_h}\frac{b(\boldsymbol{\psi}_h,P_h p-p_h)}{\|\boldsymbol{\psi}_h\|_X}\\
\lesssim &\|P_h p-p\|+h^{1/2}\inf_{\phi_h\in\Lambda_h}\|p^d-\phi_h\|_{\Gamma}+\mu\|\boldsymbol{u}-\boldsymbol{u}_h\|_X.
\end{aligned}
\end{align}
From the triangle inequality, (\ref{eqn:Ph p-p}), the projection property for $\|P_h p-p\|$, a discussion similar to (\ref{eqn:cosistency error p}) for $h^{1/2}\inf_{\phi_h\in\Lambda_h}\|p^d-\phi_h\|_{\Gamma}$, and (\ref{eqn:error estimate regular pressure depend}),  we get (\ref{eqn:error estimate of p regular pressure depend}) and complete the proof.
\end{proof}

\begin{remark}
The pressure dependency come from the presence of a function in $V_h(0)$, which either does not satisfy exactly divergence-free or does not satisfy interface normal continuity.
\end{remark}
\section{Pressure-robust discretization}
In this section,  we will apply a reconstruction operator to the discretization scheme to achieve exactly divergence-free and interface normal continuity, thereby obtaining a pressure-robust discretization. Therefore, for any $T\in\mathcal{T}_h$, we consider the Raviart-Thomas space $RT_{k-1}(T)=[P_{k-1}(T)]^N+\boldsymbol{x}P_{k-1}(T)$ for $k\geq 2$. We get for $\boldsymbol{v}_h\in RT_{k-1}(T)$ that $\nabla\cdot\boldsymbol{v}_h\in P_{k-1}(T)$,
$\boldsymbol{v}_h\cdot\boldsymbol{n}|_e\in P_{k-1}(e), e\subset\partial T$.
Then, define
\begin{align*}
\Theta_h^i=\{\boldsymbol{v}_h\in H(div;\Omega^i)~|~\boldsymbol{v}_{h|T}\in RT_{k-1}(T),\forall T\in\mathcal{T}_h, \boldsymbol{v}_h\cdot \boldsymbol{n}=0 \mbox{~on~}\Gamma^i\},
\end{align*}
with $i=s,d$, respectively. 
Note that $\Theta_h^d=V_h^d$.
Then, we define reconstruction operators $\Pi^i_h: V^i\rightarrow \Theta_h^i$ locally by 
\begin{align}
\label{eqn:Pi 1}
&(\Pi_h^i\boldsymbol{v}-\boldsymbol{v},\boldsymbol{\psi}_h)_T=0,\quad \boldsymbol{\psi}_h\in [P_{k-2}(T)]^N,\\
\label{eqn:Pi 2}
&\langle(\Pi_h^i\boldsymbol{v}-\boldsymbol{v})\cdot\boldsymbol{n},q_h\rangle_e=0,\quad
q_h\in P_{k-1}(e),
e\subset \partial T,
\end{align}
for any element  $T\subset\Omega^i$. From (2.5.10) and Proposition 2.5.1 in \cite{Boffi2013}, the reconstruction operators are well defined with the property
\begin{align}\label{eqn:Pi inequality}
\|\boldsymbol{v}-\Pi_h^i\boldsymbol{v}\|_{\ell,T}\lesssim h_T^{m-\ell}|\boldsymbol{v}|_{m,T},\quad  \ell=0,1,
\end{align}
for any $\boldsymbol{v}\in H^m(\Omega^i)$ and $1\leq m\leq k$.

Next, we rewrite the Lemma~4.1 in \cite{Zhang2025} into a more detailed conclusion. The proof is the same as the original one. To maintain the integrity of the text, we still provide the proof in Appendix.
\begin{lemma}\label{lem:Pi}
Define $\Pi_h=\Pi_h^s\times \Pi_h^d: V^s\times V^d\rightarrow \Theta_h=\Theta_h^s\times \Theta_h^d$, which
has the following properties
\begin{align}
&\Pi_h:V_h\rightarrow \Theta_h\cap \Theta_b,\\
&\Pi_h:V_h(0)\rightarrow \Theta_h\cap \Theta_d\cap \Theta_b.
\end{align}
where
\begin{align*}
\Theta_b&=\{\boldsymbol{\psi}\in H(div,\Omega)~|~  
(\boldsymbol{\psi} _s\cdot\boldsymbol{n}^s)_{|\Gamma^s}=0, ~
(\boldsymbol{\psi} _d\cdot\boldsymbol{n}^d)_{|\Gamma^d}=0,  \mbox{~and~}(\boldsymbol{\psi}^s\cdot\boldsymbol{n}^s
+\boldsymbol{\psi}^d\cdot\boldsymbol{n}^d)_{|\Gamma}=0\},\\
\Theta_d&=\{\boldsymbol{\psi}\in H(div,\Omega)~|~ \nabla\cdot\boldsymbol{\psi}=0\}.
\end{align*}
\end{lemma}
\begin{proof}
The proof can be found in Appendix.
\end{proof}

\begin{lemma}
For any $\boldsymbol{\psi}_h^s\in V_h^s$ and $\boldsymbol{u}^s\in [H^2(\Omega^s)]^N$, we have the following estimates
\begin{align}
\label{eqn:div 1-Pi}
\begin{aligned}
&2\mu(\nabla\cdot D(\boldsymbol{u}^s),(1-\Pi_h^s)\boldsymbol{\psi}_h^s)_s\\
\lesssim &\mu\left(h^2\sum_{T\in\mathcal{T}_h(\Omega^s)}\inf_{\boldsymbol{\varphi}_h^s\in [P_{k-2}(T)]^N}\|\nabla\cdot D(\boldsymbol{u}^s)-\boldsymbol{\varphi}_h^s\|_T^2\right)^{1/2}
\|\nabla\boldsymbol{\psi}_h^s\|_s,\\
\end{aligned}
\end{align}
\begin{align}
\label{eqn:Du n 1-Pi}
\begin{aligned}
&2\mu\langle D(\boldsymbol{u}^s)\boldsymbol{n}^s\cdot\boldsymbol{n}^s,
 (1-\Pi_h^s)\boldsymbol{\psi}_h^s\cdot\boldsymbol{n}^s\rangle_{\Gamma}\\
 \lesssim &\mu\left(h\sum_{e\in\mathcal{E}_h(\Gamma)}\inf_{q_h\in P_{k-1}(e)}
 \|D(\boldsymbol{u}^s)\boldsymbol{n}^s\cdot\boldsymbol{n}^s-q_h\|_e^2\right)^{1/2}
 \|\nabla\boldsymbol{\psi}_h^s\|_s.
 \end{aligned}
\end{align}
\end{lemma}
 \begin{proof}
For the first term, 
it is a direct conclusion from the orthogonality property (\ref{eqn:Pi 1}), Cauchy-Schwarz inequality, estimation (\ref{eqn:Pi inequality}).

For the second estimate, from the orthogonality property (\ref{eqn:Pi 2}) and Cauchy-Schwarz inequality, we have
 \begin{align*}
 &2\mu\langle D(\boldsymbol{u}^s)\boldsymbol{n}^s\cdot\boldsymbol{n}^s,
 (1-\Pi_h^s)\boldsymbol{\psi}_h^s\cdot\boldsymbol{n}^s\rangle_{\Gamma}\\
 \lesssim &\mu
 \sum_{e\in\mathcal{E}_h(\Gamma)}\inf_{q_h\in P_{k-1}(e)}
 \|D(\boldsymbol{u}^s)\boldsymbol{n}^s\cdot\boldsymbol{n}^s-q_h\|_e
 \|(1-\Pi_h^s)\boldsymbol{\psi}_h^s\cdot\boldsymbol{n}^s\|_e\\
  \lesssim &\mu
 \left(\sum_{e\in\mathcal{E}_h(\Gamma)}\inf_{q_h\in P_{k-1}(e)}
 \|D(\boldsymbol{u}^s)\boldsymbol{n}^s\cdot\boldsymbol{n}^s-q_h\|_e^2\right)^{1/2}
  \left(\sum_{e\in\mathcal{E}_h(\Gamma)}
 \|(1-\Pi_h^s)\boldsymbol{\psi}_h^s\cdot\boldsymbol{n}^s\|_e^2\right)^{1/2}.
 \end{align*}
 Using trace inequality and estimation (\ref{eqn:Pi inequality}), it can be obtained
 \begin{align}
 \label{eqn:Pi_h-I estimate}
 \begin{aligned}
 &\sum_{e\in\mathcal{E}_h(\Gamma)}\|(1-\Pi_h^s)\boldsymbol{\psi}_h^s\cdot\boldsymbol{n}^s\|_e^2\\
 \lesssim &
 \sum_{T\in\mathcal{T}_h(\Gamma^s)}h_T^{-1}\|(1-\Pi_h^s)\boldsymbol{\psi}_h^s\|_T^2
 +h_T\|\nabla(1-\Pi_h^s)\boldsymbol{\psi}_h^s\|_T^2\\
 \lesssim& h
 \|\nabla\boldsymbol{\psi}_h^s\|_s^2,
 \end{aligned}
 \end{align}
where 
$\mathcal{T}_h(\Gamma^s)=\{T\in\mathcal{T}_h(\Omega^s)~|~T\cap\Gamma\neq\emptyset\}$. 
The above two inequalities imply the second estimate and we complete the proof.
 \end{proof}

To obtain a fully pressure-robust discretization, the Stokes-Darcy equations needs to be modified as follows: Find $(\boldsymbol{u}_h,p_h)\in 
V_h\times Q_h$ solving 
\begin{align}
\label{eqn:fully pressure-robust discretization}
\begin{aligned}
 a(\boldsymbol{u}_h,\boldsymbol{\psi}_h)+b(\boldsymbol{\psi}_h,p_h)
&=(f,\Pi_h\boldsymbol{\psi}_h), && \forall
\boldsymbol{\psi}_h\in V_h,\\
b(\boldsymbol{\psi}_h,\phi_h)&=(g,\phi_h), && \forall \phi_h\in Q_h.
\end{aligned}
\end{align}
The analysis  is based on an improved pressure-related functional with reconstruction operator on $V_h$, which achieves exact divergence-free and  continuity across $\Gamma$. This functional is defined by
\begin{align}
\label{eqn:consistency error functional of pressure robust}
\aleph_p(\boldsymbol{\psi}_h)=b(\Pi_h\boldsymbol{\psi}_h,p)
-\langle\Pi_h^s\boldsymbol{\psi}_h^s\cdot\boldsymbol{n}^s
+\Pi_h^d\boldsymbol{\psi}_h^d\cdot\boldsymbol{n}^d,p^d\rangle_{\Gamma}
, \quad \forall \boldsymbol{\psi}_h\in V_h,
\end{align}
related to $p^d\in H^1(\Omega^d)$. This operator was first introduced in our previous work~\cite{Zhang2025} for the Stokes-Darcy optimal control problem as a functional measuring the violation of the discrete divergence constraint and interface normal continuity, and we adopt the same construction here for the coupled Stokes-Darcy problem. Lemma~\ref{lem:Pi} implies $\aleph_p(\boldsymbol{\psi}_h)=0$ for any $\boldsymbol{\psi}_h\in V_h(0)$ and any $p\in Q$ with $p^d\in H^1(\Omega^d)$, which leads to the discrete format being pressure-independent.

\begin{lemma}\label{lem:error uh-Shu}
If the exact solutions of (\ref{eqn:state weak 1}) and (\ref{eqn:state weak 2}) have the regularity
$\boldsymbol{u}^s\in [H^2(\Omega^s)]^N$, and $p^i\in H^1(\Omega^i), i=s,d$,
it holds
\begin{align}
\label{eqn:error estmate pressure robust}
\begin{aligned}
\|S_h\boldsymbol{u}-\boldsymbol{u}_h\|_X^2
\lesssim &h^2\sum_{T\in\mathcal{T}_h(\Omega^s)}\inf_{\boldsymbol{\varphi}_h^s\in [P_{k-2}(T)]^N}\|\nabla\cdot D(\boldsymbol{u}^s)-\boldsymbol{\varphi}_h^s\|_T^2\\
&+h\sum_{e\in\mathcal{E}_h(\Gamma)}\inf_{q_h\in P_{k-1}(e)}
 \|D(\boldsymbol{u}^s)\boldsymbol{n}^s\cdot\boldsymbol{n}^s-q_h\|_e^2,
\end{aligned}
\end{align}
for the solution $\boldsymbol{u}_h$ of (\ref{eqn:fully pressure-robust discretization}) and the discrete approximation $S_h\boldsymbol{u}$ of the 
exact solutions of (\ref{eqn:state weak 1}) and (\ref{eqn:state weak 2}).
\end{lemma}
\begin{proof}
From (\ref{eqn:fully pressure-robust discretization}) and (\ref{eqn:Sh definition}), it is easy to get
\begin{align}\label{eqn:error eqn uh-Shu}
a(\boldsymbol{u}_h-S_h\boldsymbol{u},\boldsymbol{\psi}_h)
+b(\boldsymbol{\psi}_h,p_h-\delta_h)
&=\mathcal{F}(\boldsymbol{\psi}_h),\quad \boldsymbol{\psi}_h\in V_h,
\end{align}
where
\begin{align*}
\mathcal{F}(\boldsymbol{\psi}_h)
=(\boldsymbol{f},\Pi_h\boldsymbol{\psi}_h)
-a(\boldsymbol{u},\boldsymbol{\psi}_h),\quad \boldsymbol{\psi}_h\in V_h.
\end{align*} 

Since $\boldsymbol{u}^s\in [H^2(\Omega^s)]^N$ and $p^i\in H^1(\Omega^i), i=s,d$, the solutions $(\boldsymbol{u},p)$ of (\ref{eqn:state weak 1}) and (\ref{eqn:state weak 2}) satisfy (\ref{eqn:state Stokes})$\sim$(\ref{eqn:state boundary 3}).
For any $\boldsymbol{\psi}_h\in V_h$, 
combining (\ref{eqn:state Stokes}) and (\ref{eqn:state Darcy}), we have
\begin{align}
\label{eqn:f equality}
\begin{aligned}
(\boldsymbol{f},\Pi_h\boldsymbol{\psi}_h)
=&-(2\mu\nabla\cdot D(\boldsymbol{u}^s),\Pi_h^s\boldsymbol{\psi}_h^s)_s
+\mu(K^{-1}\boldsymbol{u}^d,\Pi_h^d\boldsymbol{\psi}_h^d)_d\\
&+(\nabla p^s,\Pi_h^s\boldsymbol{\psi}_h^s)_s
+(\nabla p^d,\Pi_h^d\boldsymbol{\psi}_h^d)_d.
\end{aligned}
\end{align}
For the last two terms in above equation, using integration by parts, (\ref{eqn:state boundary 2}), and 
(\ref{eqn:consistency error functional of pressure robust}), it can be obtained
\begin{align}
\label{eqn: ps pd and Pi}
\begin{aligned}
&(\nabla p^s,\Pi_h^s\boldsymbol{\psi}_h^s)_s
+(\nabla p^d,\Pi_h^d\boldsymbol{\psi}_h^d)_d\\
=&b(\Pi_h\boldsymbol{\psi}_h,p)
+\langle p^s,\Pi_h^s\boldsymbol{\psi}_h^s\cdot\boldsymbol{n}^s\rangle_{\Gamma}
+\langle p^d,\Pi_h^s\boldsymbol{\psi}_h^s\cdot\boldsymbol{n}^d\rangle_{\Gamma}\\
=&b(\Pi_h\boldsymbol{\psi}_h,p)
+\langle 2\mu D(\boldsymbol{u}^s)\boldsymbol{n}^s\cdot\boldsymbol{n}^s,\Pi_h^s\boldsymbol{\psi}_h^s\cdot\boldsymbol{n}^s\rangle_{\Gamma}
+\langle p^d,\Pi_h^s\boldsymbol{\psi}_h^s\cdot\boldsymbol{n}^s+\Pi_h^d\boldsymbol{\psi}_h^d\cdot\boldsymbol{n}^d\rangle_{\Gamma}\\
=&\aleph_p(\boldsymbol{\psi}_h)
+\langle 2\mu D(\boldsymbol{u}^s)\boldsymbol{n}^s\cdot\boldsymbol{n}^s,
\Pi_h^s\boldsymbol{\psi}_h^s\cdot\boldsymbol{n}^s\rangle_{\Gamma}.
\end{aligned}
\end{align}
And from (\ref{eqn:u and psi_h}), (\ref{eqn:the first summand}), (\ref{eqn:the third summand}), and (\ref{eqn:the fourth summand}), it holds that
\begin{align}
\label{eqn:a equality}
a(\boldsymbol{u},\boldsymbol{\psi}_h)=-2\mu(\nabla\cdot D(\boldsymbol{u}^s),\boldsymbol{\psi}_h^s)_s+\mu(K^{-1}\boldsymbol{u}^d,\boldsymbol{\psi}_h^d)_d+2\mu\langle D(\boldsymbol{u}^s)\boldsymbol{n}^s\cdot\boldsymbol{n}^s
,\boldsymbol{\psi}_h^s\cdot\boldsymbol{n}^s\rangle_{\Gamma}.
\end{align}

Due to $\Theta_h^d=V_h^d$, then $\Pi_h^d$ is an identity operator for $\boldsymbol{\psi}_h^d\in V_h^d$, i.e. $(1-\Pi_h^d)\boldsymbol{\psi}_h^d=0$. Then, combining this identity with (\ref{eqn:f equality}), (\ref{eqn: ps pd and Pi}), and (\ref{eqn:a equality}), we get
\begin{align}\label{eqn:F functional}
\begin{aligned}
\mathcal{F}(\boldsymbol{\psi}_h)
=&\aleph_p(\boldsymbol{\psi}_h)+(2\mu\nabla\cdot D(\boldsymbol{u}^s),\Pi_h^s\boldsymbol{\psi}_h^s-\boldsymbol{\psi}_h^s)_s\\
&+\langle 2\mu D(\boldsymbol{u}^s)\boldsymbol{n}^s\cdot\boldsymbol{n}^s,
(\Pi_h^s\boldsymbol{\psi}_h^s-\boldsymbol{\psi}_h^s)
\cdot\boldsymbol{n}^s\rangle_{\Gamma},
\end{aligned}
\end{align}
where $\aleph_p(\boldsymbol{\psi}_h)$ is the pressure-related term. 
Set $\boldsymbol{\psi}_h=\boldsymbol{u}_h-S_h\boldsymbol{u}\in V_h(0)$ in (\ref{eqn:error eqn uh-Shu}), from Lemma~\ref{lem:Pi}, it can be obtained $\aleph_p(\boldsymbol{u}_h-S_h\boldsymbol{u})=0$ and $b(\boldsymbol{u}_h-S_h\boldsymbol{u},p_h-\delta_h)=0$. Then, (\ref{eqn:F functional}), (\ref{eqn:div 1-Pi}), and (\ref{eqn:Du n 1-Pi}) imply the conclusion. 
\end{proof}

\begin{theorem}\label{the:pressure-robust}
If the exact solutions of (\ref{eqn:state weak 1}) and (\ref{eqn:state weak 2}) have the regularity
$(\boldsymbol{u}^s,\boldsymbol{u}^d)\in [H^r(\Omega^s)]^N\times [H^{r-1}(\Omega^d)]^N$, $\nabla\cdot \boldsymbol{u}^d\in H^{r-1}(\Omega^d)$, and $p^i\in H^{r-1}(\Omega^i), i=s,d$ for any $k\leq r\leq k+1$ with $k\geq 2$,
it holds
\begin{align}\label{eqn:error estiamte regular for pressure robust}
\|\boldsymbol{u}-\boldsymbol{u}_h\|_X^2
\lesssim &h^{2(r-1)}(\|\boldsymbol{u}^s\|_{r,s}^2
+\|\boldsymbol{u}^d\|_{r-1,d}^2+\|\nabla\cdot\boldsymbol{u}^d\|_{r-1,d}^2),
\end{align}
and
\begin{align}\label{eqn:error estimate of p regular pressure depend for robust method}
\|p-p_h\|^2
\lesssim &h^{2(r-1)}(\mu^2\|\boldsymbol{u}^s\|_{r,s}^2+\mu^2\|\boldsymbol{u}^d\|_{r-1,d}^2+\mu^2\|\nabla\cdot\boldsymbol{u}^d\|_{r-1,d}^2
+\|p^s\|_{r-1,s}^2+\|p^d\|_{r-1,d}^2),
\end{align}
for the solution $(\boldsymbol{u}_h,p_h)$ of (\ref{eqn:fully pressure-robust discretization}).
\end{theorem}
\begin{proof}
Under the regularity $(\boldsymbol{u}^s,\boldsymbol{u}^d)\in [H^r(\Omega^s)]^N\times [H^{r-1}(\Omega^d)]^N$ and $\nabla\cdot \boldsymbol{u}^d\in H^{r-1}(\Omega^d)$ for any $k\leq r \leq k+1$ with $k\geq 2$, for the first term in (\ref{eqn:error estmate pressure robust}), from projection property, we have
\begin{align}\label{eqn:first two term estimate}
\begin{aligned}
h^2\sum_{T\in\mathcal{T}_h(\Omega^s)}\inf_{\boldsymbol{\varphi}_h^s\in [P_{k-2}(T)]^N}\|\nabla\cdot D(\boldsymbol{u}^s)-\boldsymbol{\varphi}_h^s\|_T^2
\lesssim h^{2(r-1)}\|\boldsymbol{u}^s\|_{r,s}^2.
\end{aligned}
\end{align}
For the second term in (\ref{eqn:error estmate pressure robust}), let $\boldsymbol{\Psi}_h^s$ to be the local $L^2$ projection of $D(\boldsymbol{u}^s)$ into $\{\boldsymbol{\Phi}_h\in [L^2(\Omega^s)]^{N\times N}~|~\boldsymbol{\Phi}_{h|T}\in [P_{k-1}(T)]^{N\times N},~T\in \mathcal{T}_h(\Omega^s)\}$. From trace inequality and projection property, it is obtained
\begin{align}\label{eqn:third and fourth term estimate}
\begin{aligned}
&h\sum_{e\in\mathcal{E}_h(\Gamma)}\inf_{q_h\in P_{k-1}(e)}
\|D(\boldsymbol{u}^s)\boldsymbol{n}^s\cdot\boldsymbol{n}^s-q_h\|_e^2\\
\lesssim &h\sum_{e\in\mathcal{E}_h(\Gamma)}
\|D(\boldsymbol{u}^s)\boldsymbol{n}^s\cdot\boldsymbol{n}^s-\boldsymbol{\Psi}_h^s\boldsymbol{n}^s\cdot\boldsymbol{n}^s\|_e^2\\
\lesssim &h\sum_{T\in\mathcal{T}_h(\Gamma^s)}
h_T^{-1}\|D(\boldsymbol{u}^s)-\boldsymbol{\Psi}_h^s\|_T^2+h_T|D(\boldsymbol{u}^s)-\boldsymbol{\Psi}_h^s|_{1,T}^2\\
\lesssim &h^{2(r-1)}\|\boldsymbol{u}^s\|_{r,s}^2,
\end{aligned}
\end{align}
where $\mathcal{T}_h(\Gamma^s)$ is defined in (\ref{eqn:Pi_h-I estimate}). From the triangle inequality, Lemma~\ref{lem:error uh-Shu}, (\ref{eqn:approximation error u and z}), (\ref{eqn:error estmate pressure robust}), (\ref{eqn:first two term estimate}), and (\ref{eqn:third and fourth term estimate}), we get (\ref{eqn:error estiamte regular for pressure robust}).

Next, we will show (\ref{eqn:error estimate of p regular pressure depend for robust method}) with the regularity
$(\boldsymbol{u}^s,\boldsymbol{u}^d)\in [H^r(\Omega^s)]^N\times [H^{r-1}(\Omega^d)]^N$, $\nabla\cdot \boldsymbol{u}^d\in H^{r-1}(\Omega^d)$, and $p^i\in H^{r-1}(\Omega^i), i=s,d$ for any $k\leq r\leq k+1$ with $k\geq 2$. For any $\boldsymbol{\psi}_h\in V_h$, 
with the help of (\ref{eqn:F functional}), noting $(1-\Pi_h^d)\boldsymbol{\psi}_h^d=0$ and $\aleph_p(\boldsymbol{\psi}_h)=b(\Pi_h\boldsymbol{\psi}_h,p)$ (from Lemma~\ref{lem:Pi}, $\langle\Pi_h^s\boldsymbol{\psi}_h^s\cdot\boldsymbol{n}^s
+\Pi_h^d\boldsymbol{\psi}_h^d\cdot\boldsymbol{n}^d,p^d\rangle_{\Gamma}=0$), it can be obtained
\begin{align}
\label{eqn:b Ph p-p}
\begin{aligned}
&b(\boldsymbol{\psi}_h,P_h p-p_h)\\
=&b(\boldsymbol{\psi}_h,P_h p)+a(\boldsymbol{u}_h,\boldsymbol{\psi}_h)-(\boldsymbol{f},\Pi_h\boldsymbol{\psi}_h)\\
=&a(\boldsymbol{u},\boldsymbol{\psi}_h)+b(\boldsymbol{\psi}_h,P_h p)-(\boldsymbol{f},\Pi_h\boldsymbol{\psi}_h)-a(\boldsymbol{u}-\boldsymbol{u}_h,\boldsymbol{\psi}_h)
\\
=&(2\mu\nabla\cdot D(\boldsymbol{u}^s),\Pi_h^s\boldsymbol{\psi}_h^s-\boldsymbol{\psi}_h^s)_s
+\langle 2\mu D(\boldsymbol{u}^s)\boldsymbol{n}^s\cdot\boldsymbol{n}^s,(\boldsymbol{\psi}_h^s-\Pi_h^s\boldsymbol{\psi}_h^s)\cdot\boldsymbol{n}^s\rangle_{\Gamma}\\
&+b(\boldsymbol{\psi}_h,P_h p-p)+b(\boldsymbol{\psi}_h-\Pi_h\boldsymbol{\psi}_h,p)-a(\boldsymbol{u}-\boldsymbol{u}_h,\boldsymbol{\psi}_h),
\end{aligned}
\end{align}
where $P_h$ is defined in (\ref{eqn:b operator equality}). 
Using the projection property and noting $\nabla\cdot\boldsymbol{\psi}_h^d\in Q_h$, the third term in (\ref{eqn:b Ph p-p}) holds
\begin{align}
\label{eqn:third term pressure p}
\begin{aligned}
b(\boldsymbol{\psi}_h,P_h p-p)=&(\nabla\cdot\boldsymbol{\psi}_h^s,(P_h p-p)^s)_s
+(\nabla\cdot\boldsymbol{\psi}_h^d,(P_h p-p)^d)_d\\
\lesssim& 
h^{r-1}\|p^s\|_{r-1,s}\|\nabla\boldsymbol{\psi}_h\|_X.
\end{aligned}
\end{align}
For the fourth term in (\ref{eqn:b Ph p-p}), with the help of the orthogonality property (\ref{eqn:Pi 1}) and (\ref{eqn:Pi 2}), the estimations (\ref{eqn:Pi inequality}) and  (\ref{eqn:Pi_h-I estimate}), a discussion similar to (\ref{eqn:cosistency error p}) for $\inf_{\phi_h\in\Lambda_h}\|p^s-\phi_h\|_{\Gamma}$, and the projection property, it can be obtained
\begin{align}
\label{eqn:fourth term pressure p}
\begin{aligned}
&b(\boldsymbol{\psi}_h-\Pi_h\boldsymbol{\psi}_h,p)\\
=&(\nabla\cdot(\boldsymbol{\psi}_h^s-\Pi_h\boldsymbol{\psi}_h^s),p^s)_s+
(\nabla\cdot(\boldsymbol{\psi}_h^d-\Pi_h\boldsymbol{\psi}_h^d),p^d)_d\\
=&(\nabla\cdot(\boldsymbol{\psi}_h^s-\Pi_h\boldsymbol{\psi}_h^s),p^s)_s\\
=&-(\boldsymbol{\psi}_h^s-\Pi_h\boldsymbol{\psi}_h^s,\nabla p^s)_s+\langle (\boldsymbol{\psi}_h^s-\Pi_h\boldsymbol{\psi}_h^s)\cdot\boldsymbol{n}^s,p^s\rangle_{\Gamma}\\
\lesssim &\|\boldsymbol{\psi}_h^s-\Pi_h\boldsymbol{\psi}_h^s\|_s\sum_{T\in\mathcal{T}_h(\Omega^s)}\inf_{\Phi_h\in [P_{k-2}(T)]^N}\|\nabla p^s-\Phi_h\|_T\\
&+\sum_{e\in\mathcal{E}_h(\Gamma)}\|\boldsymbol{\psi}_h^s-\Pi_h\boldsymbol{\psi}_h^s)\|_{e}\inf_{\phi_h\in P_{k-1}(e)}\|p^s-\phi_h\|_e\\
\lesssim & h^{r-1}\|p^s\|_{r-1,s}\|\nabla\boldsymbol{\phi}_h^s\|_s.
\end{aligned}
\end{align}
Combining (\ref{eqn:b Ph p-p}), (\ref{eqn:div 1-Pi}), (\ref{eqn:Du n 1-Pi}), (\ref{eqn:first two term estimate}), (\ref{eqn:third and fourth term estimate})
(\ref{eqn:third term pressure p}), (\ref{eqn:fourth term pressure p}), and the boundedness of $a(\cdot,\cdot)$, the following estimate holds
\begin{align*}
\begin{aligned}
\|P_h p-p_h\|\lesssim &\sup_{\boldsymbol{\psi}_h\in V_h}\frac{b(\boldsymbol{\psi}_h,P_h p-p_h)}{\|\boldsymbol{\psi}_h\|_X}
\lesssim  \left(\mu h^{r-1}\|\boldsymbol{u}^s\|_{r,s}^2+h^{r-1}\|p^s\|_{r-1,s}^2+\mu\|u-u_h\|_X^2\right)^{1/2},
\end{aligned}
\end{align*}
which implies (\ref{eqn:error estimate of p regular pressure depend for robust method}) by triangle inequality $\|p-p_h\|\leq \|p-P_h p\|+\|P_h p-p_h\|$ and (\ref{eqn:error estiamte regular for pressure robust}).
\end{proof}

\section{Numerical experiments}
In this section, we present three two-dimensional numerical examples to assess the performance of the proposed method. The first two examples are manufactured problems, whereas the third one is a more involved driven cavity problem with a piecewise linear interface, for which no analytical solution is available.

The purpose of the first example is to examine the influence of pressure scaling on the velocity error: the viscosity is fixed and the pressure field is multiplied by different factors. Conversely, the second example fixes the pressure field and varies the viscosity coefficient, allowing us to investigate the sensitivity of the velocity error with respect to viscosity perturbations. These two tests highlight the contrasting behavior of classical and pressure-robust schemes.
For the first two examples, we consider two polynomial orders, namely $k=2$ and $k=3$. For $k=2$, the Stokes equations are discretized by the $\boldsymbol{P}_2^{+}$--$P_1$ element pair and the Darcy problem by the $RT_1$--$P_1$ element. For $k=3$, the corresponding discretizations are $\boldsymbol{P}_3^{+}$--$P_2$ for Stokes and $RT_2$--$P_2$ for Darcy. For the third example, we only consider the case $k=2$, as the purpose of this test is to assess robustness on a benchmark configuration rather than to examine convergence, and the higher-order case has already been tested in the manufactured examples.

All computations are carried out in \textsc{GNU Octave} (version~10.3.0) on a standard workstation. The source codes are publicly available at \texttt{https://github.com/custzjc/Stokes-Darcy-FEM}. The matrices are assembled with exact integration, and the right-hand sides and error norms are evaluated using an eighth-order accurate quadrature to ensure reliable error measurements. Lagrange multipliers are employed to impose the interface conditions \cite{riviereLocallyConservativeCoupling2005} and to enforce the zero-mean constraint for the pressure, while a penalty method is adopted for the Dirichlet boundary conditions. The resulting global linear systems are solved using the built-in backslash operator, which invokes a sparse direct LU factorization.

For convenience of presentation in the numerical plots, we abbreviate the velocity and pressure errors as follows:
\begin{align*}
E_h&=\|\boldsymbol{u}-\boldsymbol{u}_h\|_X,\quad 
E_h^s=|\boldsymbol{u}^s-\boldsymbol{u}_h^s|_{1,s},\quad
E_{h,1}^d=\|\boldsymbol{u}^d-\boldsymbol{u}_h^d\|_d,\quad
E_{h,2}^d=\|\nabla\cdot(\boldsymbol{u}^d-\boldsymbol{u}_h^d)\|_d,
\\
e_h&=\|p-p_h\|,\quad e_h^s=\|p-p_h\|_s,\quad e_h^d=\|p-p_h\|_d.
\end{align*}
\subsection{Basis functions, degree of freedom, and computing $(\boldsymbol{f},\Pi_h\boldsymbol{\psi}_h)$}
To facilitate reproducibility and to clarify the implementation of the pressure-robust formulation, we first summarize the basis functions, degrees of freedom, and the assembly of the reconstructed right-hand side for the case $k=2$, which is used in all benchmark tests. The extension to $k=3$ is straightforward, since the divergence-free reconstruction operator and the associated coupling terms remain unchanged and only the local polynomial degrees differ.
The $\boldsymbol{P}_2^{+}$ and $P_1$ spaces utilize standard Lagrange basis functions supplemented with bubble functions. For the Raviart-Thomas ($RT_1$) elements, we use the Bernstein-Bézier local basis functions as detailed in \cite{ainsworthBernsteinBezierBasis2015}. The specific basis functions construction for any element $T$ are given as follows:
\begin{align*}
&\frac{1}{2|T|}(\lambda_2|e_3|\boldsymbol{t}_3-\lambda_3|e_2|\boldsymbol{t}_2),\quad
\frac{1}{2|T|}(\lambda_3|e_1|\boldsymbol{t}_1-\lambda_1|e_3|\boldsymbol{t}_3),\quad
\frac{1}{2|T|}(\lambda_1|e_2|\boldsymbol{t}_2-\lambda_2|e_1|\boldsymbol{t}_1),\\
&\frac{1}{|T|}(\lambda_3|e_2|\boldsymbol{t}_2+\lambda_2|e_3|\boldsymbol{t}_3),\quad
\frac{1}{|T|}(\lambda_3|e_1|\boldsymbol{t}_1+\lambda_1|e_3|\boldsymbol{t}_3),\quad
\frac{1}{|T|}(\lambda_2|e_1|\boldsymbol{t}_1+\lambda_1|e_2|\boldsymbol{t}_2),\\
&\frac{1}{|T|}(\lambda_1\lambda_2|e_3|\boldsymbol{t}_3-\lambda_1\lambda_3|e_2|\boldsymbol{t}_2),\quad
\frac{1}{|T|}(\lambda_1\lambda_3|e_2|\boldsymbol{t}_2-\lambda_2\lambda_3|e_1|\boldsymbol{t}_1).
\end{align*}
In this context, let $e_j$ and $\boldsymbol{t}_j$ ($j=1,2,3$) represent the $j$-th edge and its corresponding unit tangent vector, respectively. The barycentric coordinates are denoted by $\{\lambda_j\}_{j=1}^3$, while $|T|$ and $|e_j|$ indicate the area (or length) of element $T$ and edge $e_j$, respectively.

For the convenience of programming implementation, we will continue to provide the degree of freedom ($d.o.f$) for each finite element space \cite{loggAutomatedSolutionDifferential2012}. Let $S_p, S_s$, and $S_e$ represent the number of points, edges, and elements of $\mathcal{T}_h(\Omega^s)$, respectively. The $d.o.f$ for any component  function $v_h$ in $V_h^s$ with $k=2$ is defined by point evaluation
\begin{align*}
\left\{
\begin{aligned}
&v_h(x_p^j), \quad j=1,\cdots,S_p,\\
&v_h(x_m^j), \quad j=1,\cdots,S_s,\\
&v_h(x_T)-\sum_{x_p^j\subset T} v_h(x_p^j)-\sum_{x_m^j\subset T} v_h(x_m^j), \quad T\in\mathcal{T}_h(\Omega^s).
\end{aligned}
\right.
\end{align*}
where $\{x_p^j\}_{j=1}^{S_p}$, $\{x_m^j\}_{j=1}^{S_s}$, and $\{x_T\}_{T\in\mathcal{T}_h(\Omega^s)}$ are the enumerations of points in $\mathcal{T}_h(\Omega^s)$, midpoints of edges in $\mathcal{E}_h(\Omega^s)$, and centroid of elements in $\mathcal{T}_h(\Omega^s)$, respectively. The dimension of $V_h^s$ is 
$2(S_p+S_s+S_e)$ in 2-dimensional case.
If we expand the function $v_h$ into the form of a linear combination of bases, the coefficients corresponding to the quadratic Lagrangian basis function are $v_h(x_p^j)$ and $v_h(x_m^j)$, respectively. The third point evaluation in above equation is the coefficients corresponding to the bubble basis functions ($27b_T$) in element $T$. 

For the Darcy domain and $k=2$, the finite element space $V_h^d$ is the $RT_1$ with $d.o.f$ defined by
\begin{align*}
\left\{
\begin{aligned}
&\int_e \boldsymbol{v}_h\cdot\boldsymbol{n} q_h, \quad q_h \mbox{~is the basis function of~} P_1(e),  e\in\mathcal{E}_h(\Omega^d),\\
&\int_T \boldsymbol{v}_h\cdot\boldsymbol{\psi}_h,\quad \boldsymbol{\psi}_h \mbox{~is the basis function of~} [P_0(T)]^2, T\in\mathcal{T}_h(\Omega^d).
\end{aligned}
\right.
\end{align*}
for any function $\boldsymbol{v}_h\in V_h^d$. Let $D_s$ and $D_e$ represent the number of edges and elements of $\mathcal{T}_h(\Omega^d)$, respectively. The dimension of $RT_1$ is $2(D_s+D_e)$.
  
The pressure space $Q_h$ with $k-1=1$ is the first-order discontinuous Lagrange element. For any function $q_h$ in $Q_h$, the $d.o.f$ can be defined by
\begin{align*}
q_h(x_T^j),  \quad j=1,2,3, \mbox{~and~} T\in\mathcal{T}_h,
\end{align*}
where $\{x_T^j\}_{j=1}^3$ is an enumerations of vertices in $T$. The dimension of $Q_h$ is $3N_e$, where $N_e$ is the number of elements of $\mathcal{T}_h$.

When programming to implement pressure-robust method, only the right-hand side of the final linear system differs from classical method. In detail, the classical and pressure-robust methods to construct right-hand sides based on $(\boldsymbol{f},\boldsymbol{\psi}_j)$ and $(\boldsymbol{f},\Pi_h\boldsymbol{\psi}_j)$, respectively, where $\boldsymbol{\psi}_j$ is the basis functions in $V_h$. For the term $(\boldsymbol{f},\boldsymbol{\psi}_j)$, a traditional quadrature formula is enough by obtaining the values of the integration points with respect to $\boldsymbol{f}$ and $\boldsymbol{\psi}_j$, respectively. But for the term $(\boldsymbol{f},\Pi_h\boldsymbol{\psi}_j)$, it becomes more complicated to get the values of the integration points with respect to $\Pi_h\boldsymbol{\psi}_j$. For any element $T\in\mathcal{T}_h$, let $\{\boldsymbol{\chi}_{\ell}\}_{\ell=1}^8$ be the basis functions in $RT_1$ related to $T$. Then, we expand $\Pi_h\boldsymbol{\psi}_{j|T}$ into the form of a linear combination of bases, i.e., $\Pi_h\boldsymbol{\psi}_{j|T}=\sum_{\ell=1}^8 w_{\ell}\boldsymbol{\chi}_{\ell}$. The  coefficients $w_{\ell}$ will be determined by
\begin{align*}
\left\{
\begin{aligned}
&(\Pi_h\boldsymbol{\psi}_j\cdot\boldsymbol{n}, \phi_h)_e=(\boldsymbol{\psi}_j\cdot\boldsymbol{n}, \phi_h)_e, \quad\forall \phi_h\in P_1(e), e\subset\partial T,\\
&(\Pi_h\boldsymbol{\psi}_j,\boldsymbol{\phi}_h)_T=(\boldsymbol{\psi}_j,\boldsymbol{\phi}_h)_T,\quad \forall \boldsymbol{\phi}_h\in [P_0(T)]^2.
\end{aligned}
\right.
\end{align*}
We also note that $\boldsymbol{\psi}_j$ has vanishing boundary values on $\Gamma^i$, and by Lemma~4.1 the reconstruction $\Pi_h^i\boldsymbol{\psi}_j$ automatically satisfies $\Pi_h^i\boldsymbol{\psi}_j \cdot \boldsymbol{n}=0$ on $\Gamma^i$. Therefore, no additional enforcement of boundary conditions is required in the assembly.

\subsection{Example: Varying pressure magnitude}\label{ex1}
This example is taken from \cite{correaUnifiedMixedFormulation2009} and has been modified with an adjustable parameter $\gamma$ to compares the error behavior of pressure-robust method with classical method under varying pressure. The domain is $\Omega=(0,1)\times(0,1)$ with interface $\Gamma=(0,1)\times \{0.5\}$, i.e. $\Omega^s=(0,1)\times(0.5,1)$ and $\Omega^d=(0,1)\times(0,0.5)$. We consider the coupled system with the following exact solution  
\begin{align*}
&\boldsymbol{u}^s=\left[-\frac{e^{\frac{y}{2}}\sin(\pi x)}{2\pi^2},\frac{e^{\frac{y}{2}}\cos(\pi x)}{\pi}\right]^t, \quad p^s=-\frac{\gamma e^{\frac{y}{2}}\cos(\pi x)}{\pi},\\
&\boldsymbol{u}^d=\left[-2e^{\frac{y}{2}}\sin(\pi x),\frac{e^{\frac{y}{2}}\cos(\pi x)}{\pi}\right]^t,  \quad p^d=-\frac{(\gamma+1)e^{\frac{y}{2}}\cos(\pi x)}{\pi}.
\end{align*}
In (\ref{eqn:state weak 1}), the constant $\frac{\alpha_1}{\sqrt{\kappa_1}}=(1+4\pi^2)/2$  can be directly derived from the exact solutions of the system. Following this approach, we also obtain the functional expressions for $\boldsymbol{f}$ and $\boldsymbol{g}$. To evaluate the performance of both methods, we fix $\mu=1$ and $K=10^{-4}$ while selecting $\gamma$ values from the geometric sequence $\{1, 10, 10^2, 10^3, 10^4, 10^5\}$.

As demonstrated in Table~\ref{tab:time and error ex1}, a comprehensive comparison reveals that while both the classical and pressure-robust methods achieve the theoretically predicted convergence orders, their error behaviors diverge significantly with increasing $\gamma$. At $\gamma=1$, both methods exhibit nearly identical performance in both error metrics and computational time. However, for $\gamma=10^5$, the classical method demonstrates a substantially larger velocity error compared to the pressure-robust method, despite comparable computational costs. This discrepancy is further corroborated by the component-wise error decomposition in Figures~\ref{fig:convergence velocity for ex1}, \ref{fig:convergence pressure for ex1}, and \ref{fig:error R change for ex1}, which highlight that the elevated velocity error in the classical method predominantly originates from the Stokes domain. These observations collectively confirm that the velocity error of the pressure-robust method remains decoupled from pressure influences, whereas the classical method's velocity error exhibits pressure-dependent behavior.

Figure~\ref{fig:error R change for ex1} reveals an additional critical feature: both the velocity and pressure errors of the classical method, along with the pressure error of the pressure-robust method, display a distinct piecewise growth pattern with respect to $\gamma$. Specifically, errors increase gradually for $\gamma\in[1,10^2]$ but accelerate markedly for $\gamma\in[10^2,10^5]$. This bifurcated trend aligns precisely with the theoretical framework established in Theorems~\ref{the:classical estimate} and \ref{the:pressure-robust}. As indicated by the error estimates in (\ref{eqn:error estimate regular pressure depend}), (\ref{eqn:error estimate of p regular pressure depend}), and (\ref{eqn:error estimate of p regular pressure depend for robust method}), all three error components adopt a unified parametric form $C_1+\gamma C_2$, where $C_1$ and $C_2$ represent $\gamma$-independent constants. This formulation explains the observed transition between error regimes as $\gamma$ exceeds $10 ^2$.

Figures~\ref{fig:convergence P3 for ex1} and~\ref{fig:error P3 R change for ex1} show that the same conclusion holds for $k=3$ as for $k=2$; in particular, the pressure-robust formulation remains effective for higher-order finite element discretizations, both in theory and in practice.

\begin{table}[htbp]
\centering \caption{The comparison of error and time cost between classical method and pressure-robust method in Example~\ref{ex1}.} \label{tab:time and error ex1}
\begin{tabular*}{\hsize}{@{}@{\extracolsep{\fill}}c|cccccc@{}}
    \hline
   &$d.o.f$          &  $\|\boldsymbol{u}-\boldsymbol{u}_h\|_X$ & order   &$\|p-p_h\|$&order&time(s)\\
    \hline
    \hline
      classical&$319$    &  $1.016$E-1  & $-$     &$5.252$E0 &$-$     &$0.820$\\

       method&$1179$   &  $2.774$E-2  & $0.993$ &$6.834$E-1 &$1.559$ &$0.716$ \\

      $\gamma=1$&$4531$   &  $4.861$E-3  & $1.293$ &$8.505$E-2 &$1.547$ &$0.854$ \\

     &$17763$   &  $1.087$E-3  & $1.095$ &$1.073$E-2 &$1.515$ &$3.220$\\

     &$70339$   &  $2.665$E-4  & $1.021$ &$1.335$E-3 &$1.514$ &$89.42$\\
    \hline
      pressure-robust&$319$    &  $1.015$E-1  & $-$     &$5.252$E0 &$-$     &$2.436$ \\

      method&$1179$   &  $2.765$E-2  & $0.995$ &$6.834$E-1 &$1.559$ &$2.424$ \\

      $\gamma=1$&$4531$   &  $4.850$E-3  & $1.293$ &$8.506$E-2 &$1.547$ &$2.617$ \\

     &$17763$   &  $1.086$E-3  & $1.094$ &$1.073$E-2 &$1.515$ &$5.536$ \\

     &$70339$   &  $2.664$E-4  & $1.0021$ &$1.336$E-3 &$1.514$ &$96.18$\\
    \hline
    \hline
      classical&$319$    &  $7.029$E1  & $-$     &$5.957$E2 &$-$     &$0.772$\\

       method&$1179$   &  $1.919$E1  & $0.992$ &$1.423$E3 &$1.095$ &$0.793$ \\

      $\gamma=10^5$&$4531$   &  $5.253$E0  & $0.962$ &$3.476$E1 &$1.047$ &$0.933$ \\

     &$17763$   &  $1.288$E0  & $1.028$ &$8.658$E0 &$1.017$ &$3.719$\\

     &$70339$   &  $3.137$E-1  & $1.026$ &$2.169$E0 &$1.005$ &$95.66$\\
    \hline
      pressure-robust&$319$    &  $1.015$E-1  & $-$     &$5.479$E2 &$-$     &$2.598$ \\

      method&$1179$   &  $2.765$E-2  & $0.995$ &$1.343$E2 &$1.075$ &$2.572$ \\

      $\gamma=10^5$&$4531$   &  $4.850$E-3  & $1.293$ &$3.356$E1 &$1.030$ &$2.770$ \\

     &$17763$   &  $1.086$E-3  & $1.094$ &$8.379$E0 &$1.015$ &$5.977$ \\

     &$70339$   &  $2.664$E-4  & $1.021$ &$2.097$E0 &$1.006$ &$98.96$\\
    \hline
\end{tabular*}
\end{table}

\begin{figure}[htbp]
  \centering
  \includegraphics[width=0.7\textwidth]{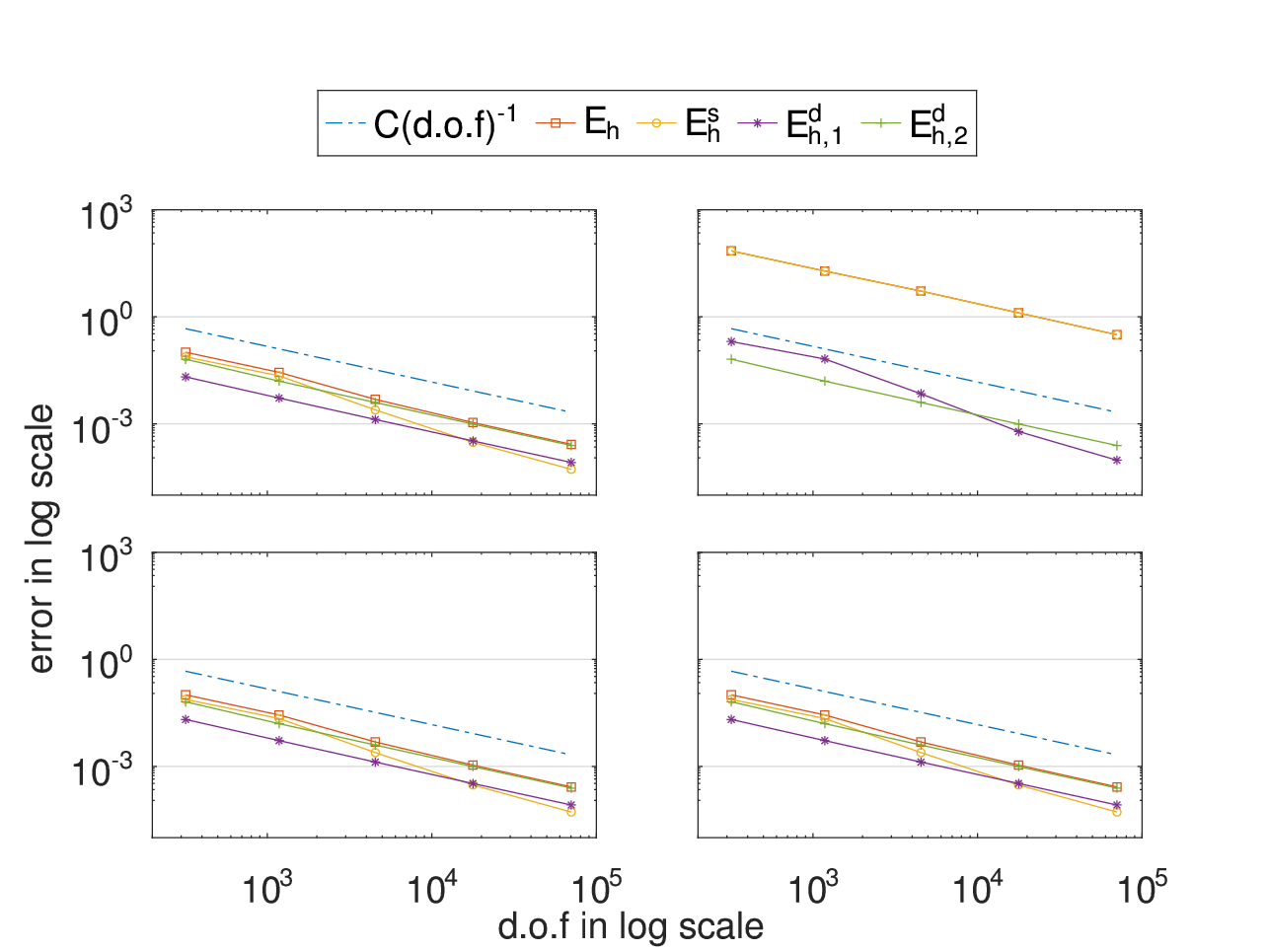}
\caption{Convergence rates of velocity for classical method (top) and pressure-robust method (bottom) with $\gamma=1$ (left) and $10^5$ (right), respectively, in Example~\ref{ex1}.}
\label{fig:convergence velocity for ex1}
\end{figure}

\begin{figure}[htbp]
  \centering
  \includegraphics[width=0.7\textwidth]{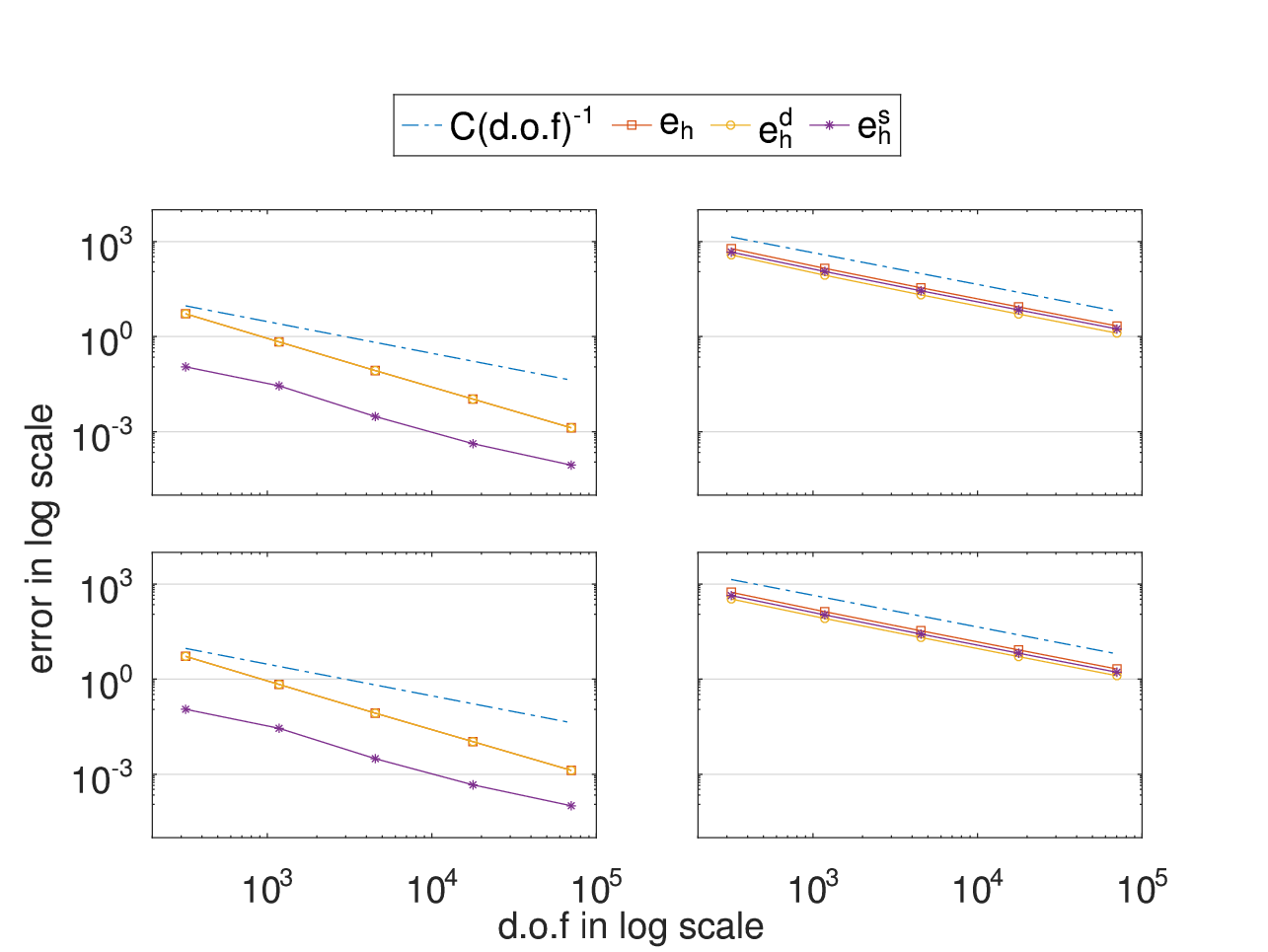}
\caption{Convergence rates of pressure for classical method (top) and pressure-robust method (bottom) with $\gamma=1$ (left) and $10^5$ (right), respectively, in Example~\ref{ex1}.}
\label{fig:convergence pressure for ex1}
\end{figure}

\begin{figure}[htbp]
  \centering
  \includegraphics[width=0.7\textwidth]{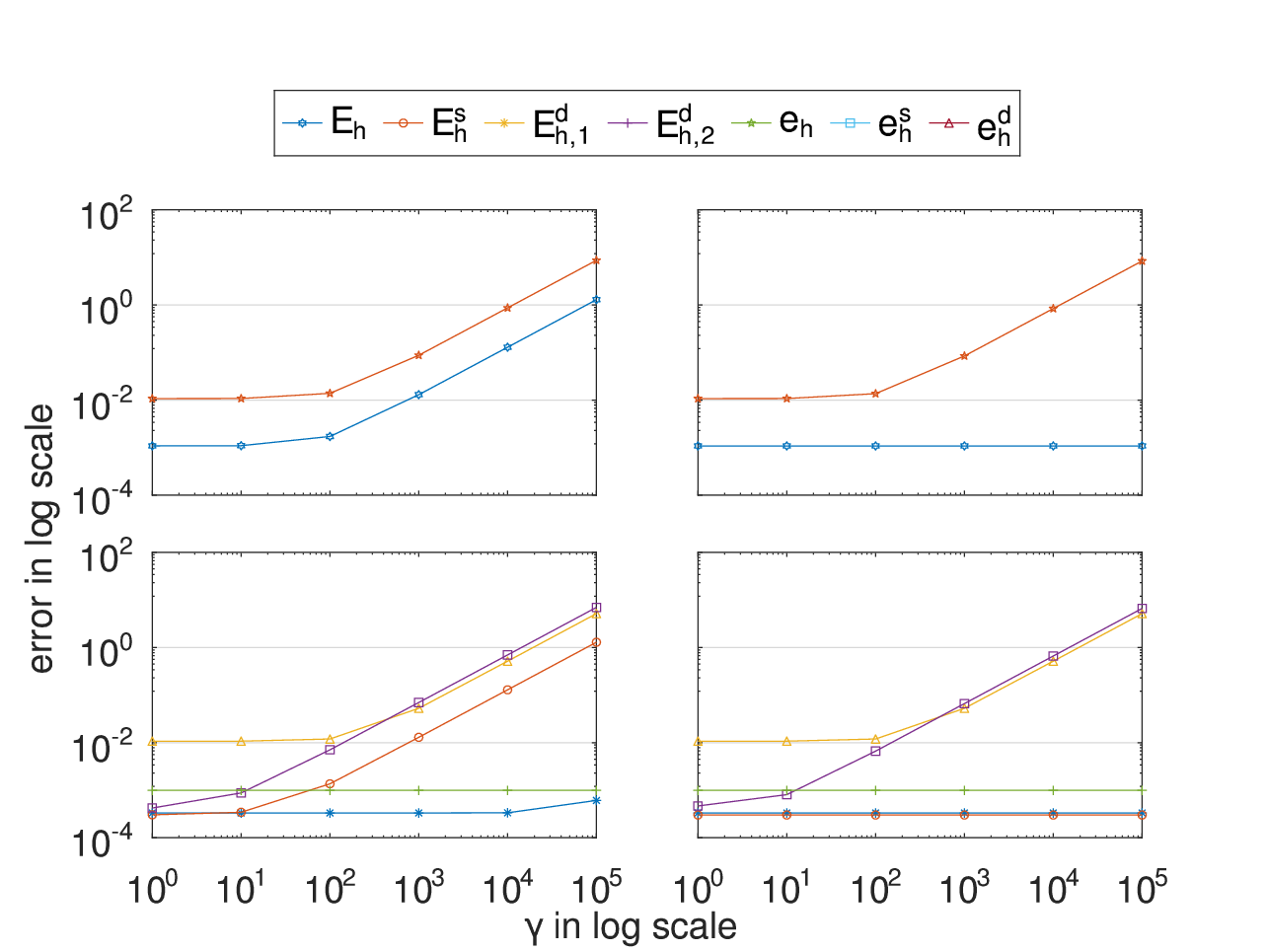}
\caption{Total errors (top) and component errors (bottom) for classical method (left) and pressure-robust method (right) with $d.o.f=17763$ in Example~\ref{ex1}.}
\label{fig:error R change for ex1}
\end{figure}

\begin{figure}[htbp]
  \centering
  \includegraphics[width=0.7\textwidth]{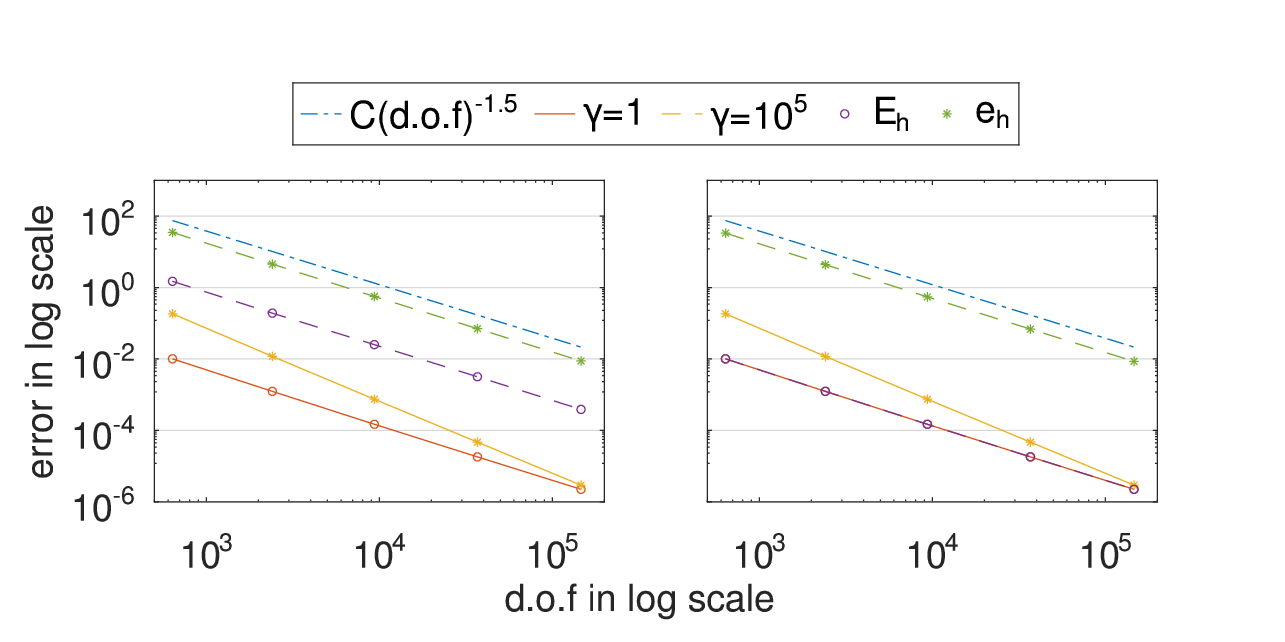}
\caption{Convergence rates for classical method (left) and pressure-robust method (right) with $k=3$ in Example~\ref{ex1}.}
\label{fig:convergence P3 for ex1}
\end{figure}

\begin{figure}[htbp]
  \centering
  \includegraphics[width=0.7\textwidth]{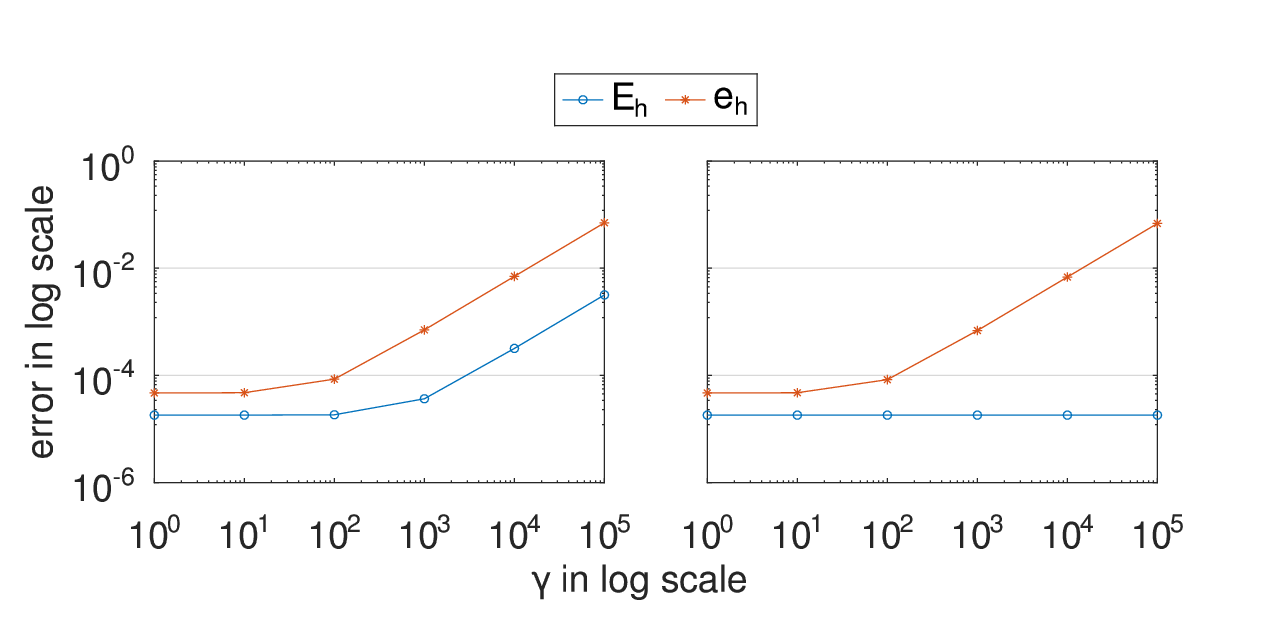}
\caption{Errors for classical method (left) and pressure-robust method (right) with $k=3$ and $d.o.f=36883$ in Example~\ref{ex1}.}
\label{fig:error P3 R change for ex1}
\end{figure}

\subsection{Example: Varying viscosity}\label{ex2}
This example from reference \cite{jiaPressurerobustWeakGalerkin2024} compares the error behavior of pressure-robust method with classical method under varying viscosity.
The Stokes and Darcy domains are $\Omega^s=(0,\pi)\times(0,\pi)$ and $\Omega^d=(0,\pi)\times(-\pi,0)$, respectively. The interface is $\Gamma: (0,\pi)\times \{0\}$. The exact solution are as follows
\begin{align*}
&\boldsymbol{u}^s=\left[2\sin y\cos y\cos x, (\sin^2 y-2)\sin x\right]^t, \quad p^s=\sin x\sin y+1669/(87\pi^2),\\
&\boldsymbol{u}^d=\left[-(e^y-e^{-y})\cos x,-(e^y+e^{-y})\sin x\right]^t,  \quad p^d=(e^y-e^{-y})\sin x+1669/(87\pi^2).
\end{align*}
This additive constant $1669/(87\pi^2)$ ensures the condition $p\in L_0^2(\Omega)$ holds. 
Under the tangential constraint $\boldsymbol{u}^s\cdot\boldsymbol{\tau}_{1|\Gamma}=-2D(\boldsymbol{u}^s)
\boldsymbol{n}^s\cdot\boldsymbol{\tau}_{1|\Gamma}=0$, the coefficient $\alpha_1$ 
in (\ref{eqn:state boundary 3}) admits any bounded value. Implementing with $K=10^{-4}$, we conduct a viscosity study ($\mu\in[10^{-6}, 10^6]$) to 
quantify $d.o.f$-convergence relationships and error performance contrasting classical and pressure-robust method. 

Focus on the error performance of velocity.
With viscosity coefficients $\mu=10^{-6},1,10^6$, Table~\ref{tab:time and error u6 ex2} demonstrate that both methods attain theoretically predicted convergence rates, while their error magnitudes exhibit different dependence on viscosity variations. In detail, when $\mu=1$ or $10^6$, Table~\ref{tab:time and error u6 ex2} shows the errors of classical and pressure-robust method have almost the same performance. However, when $\mu=10^{-6}$, classical method are affected by the viscosity, and the error in velocity is several orders of magnitude larger than that of pressure-robust method. The streamline comparison in Figure~\ref{fig:streamline for ex2} also reflects this point. This numerical phenomenon is completely consistent with theoretical analysis in (\ref{eqn:error estimate regular pressure depend}) and (\ref{eqn:error estiamte regular for pressure robust}), i.e. the error of classical method demonstrate an inverse functional dependence on the viscosity and the error of pressure-robust method is independent of viscosity. 
For a more intuitive comparison of convergence behavior and error components between the two methods, Figure~\ref{fig:convergence of velocity for ex2} comprehensively presents both total errors and their constituent elements. With the $d.o.f$ fixed at $8931$, Figure~\ref{fig:error for ex2} subsequently demonstrates the error evolution as a function of the parameter $\mu$.

Turning gaze towards the error of pressure, Table~\ref{tab:time and error u6 ex2} reveals both methods maintain comparable error magnitudes and demonstrate analogous convergence characteristics with varying viscosity. A more intuitive comparison is presented in Figure~\ref{fig:convergence of pressure for ex2}. 
With the $d.o.f$ fixed at $8931$, Figure~\ref{fig:error for ex2} demonstrates the error behavior under varying values of the parameter $\mu$. For $\mu\in[10^{-6},10^{-3}]$, the error remains nearly constant, indicating negligible dependence on $\mu$ in this regime. In contrast, as $\mu$ increases from $10^{-2}$ to $10^6$, the error grows linearly with a consistent slope. This trend aligns with the theoretical error bounds derived in (\ref{eqn:error estimate of p regular pressure depend}) and (\ref{eqn:error estimate of p regular pressure depend for robust method}), which can be concisely expressed as $\|p-p_h\|\lesssim \mu C_u+C_p$,
where $C_u$ and $C_p$ are problem-dependent constants associated with the velocity and pressure approximations, respectively.

For the case $k=3$, the results displayed in Figures~\ref{fig:convergence P3 for ex2} and~\ref{fig:error P3 mu change for ex2} show that the conclusions obtained for $k=2$ remain valid: fixing the pressure field and varying the viscosity parameter leads to a pronounced sensitivity in the classical method, whereas the pressure-robust formulation remains unaffected. This confirms that the observed robustness with respect to viscosity persists for higher-order finite element discretizations.

\begin{table}[htbp]
\centering \caption{The comparison of error and time cost between classical method and pressure-robust method in Example~\ref{ex2}.} \label{tab:time and error u6 ex2}
\begin{tabular*}{\hsize}{@{}@{\extracolsep{\fill}}c|cccccc@{}}
    \hline
   &$d.o.f$          &  $\|\boldsymbol{u}-\boldsymbol{u}_h\|_X$ & order   &$\|p-p_h\|$&order&time(s)\\
    \hline
    \hline
      classical&$167$    &  $3.607$E4  & $-$     &$2.526$E0 &$-$     &$0.763$\\

       method&$603$   &  $9.622$E3  & $1.029$ &$6.822$E-1 &$1.019$ &$0.789$ \\

      $\mu=10^{-6}$&$2291$   &  $2.309$E3  & $1.069$ &$1.748$E-1 &$1.020$ &$0.903$ \\

     &$8931$   &  $5.464$E2  & $1.059$ &$4.397$E-2 &$1.014$ &$1.134$\\

     &$35267$   &  $1.356$E2  & $1.014$ &$1.101$E-2 &$1.008$ &$4.181$\\
     
     &$140163$  &  $3.359$E1  & $1.011$ &$2.753$E-3 &$1.004$ &$126.9$ \\
    \hline
      pressure-robust&$167$    &  $5.635$E0  & $-$     &$2.524$E0 &$-$     &$2.626$ \\

      method&$603$   &  $1.577$E0  & $0.991$ &$6.818$E-1 &$1.019$ &$2.549$ \\

      $\mu=10^{-6}$&$2291$   &  $4.052$E-1  & $1.018$ &$1.746$E-1 &$1.020$ &$2.732$ \\
      
     &$8931$   &  $1.009$E-1  & $1.021$ &$4.393$E-2 &$1.014$ &$3.269$ \\

     &$35267$   &  $2.543$E-2  & $1.003$ &$1.099$E-2 &$1.008$ &$7.523$\\
     
     &$140163$  &  $6.353$E-3  & $1.005$ &$2.750$E-3 &$1.004$ &$126.6$ \\
    \hline
    \hline
      classical&$167$    &  $5.633$E0  & $-$     &$7.171$E3 &$-$     &$0.814$\\

       method&$603$   &  $1.571$E0  & $0.994$ &$1.105$E3 &$1.456$ &$0.802$ \\

      $\mu=1$&$2291$   &  $4.035$E-1  & $1.018$ &$1.413$E2 &$1.541$ &$0.890$ \\

     &$8931$   &  $1.006$E-1  & $1.020$ &$1.737$E1 &$1.540$ &$1.884$\\

     &$35267$   &  $2.537$E-2  & $1.003$ &$2.179$E0 &$1.511$ &$7.867$\\
     
     &$140163$  &  $6.339$E-3  & $1.005$ &$2.736$E-1 &$1.503$ &$432.3$ \\
    \hline
      pressure-robust&$167$    &  $5.635$E0  & $-$     &$7.171$E3 &$-$     &$2.889$\\

       method&$603$   &  $1.577$E0  & $0.991$ &$1.105$E3 &$1.456$ &$2.693$ \\

      $\mu=1$&$2291$   &  $4.052$E-1  & $1.018$ &$1.413$E2 &$1.541$ &$2.815$ \\

     &$8931$   &  $1.009$E-1  & $1.021$ &$1.737$E1 &$1.540$ &$3.928$\\

     &$35267$   &  $2.543$E-2  & $1.003$ &$2.179$E0 &$1.511$ &$9.710$\\
     
     &$140163$  &  $6.353$E-3  & $1.005$ &$2.736$E-1 &$1.503$ &$418.6$ \\
    \hline
    \hline
      classical&$167$    &  $5.633$E0  & $-$     &$7.171$E9 &$-$     &$0.898$\\

       method&$603$   &  $1.571$E0  & $0.994$ &$1.105$E9 &$1.456$ &$0.818$ \\

      $\mu=10^6$&$2291$   &  $4.035$E-1  & $1.018$ &$1.413$E8 &$1.541$ &$0.901$ \\

     &$8931$   &  $1.006$E-1  & $1.020$ &$1.737$E7 &$1.540$ &$2.321$\\

     &$35267$   &  $2.537$E-2  & $1.003$ &$2.179$E6 &$1.511$ &$18.88$\\
     
     &$140163$  &  $6.339$E-3  & $1.005$ &$2.735$E5 &$1.504$ &$1192$ \\
    \hline
      pressure-robust&$167$    &  $5.635$E0  & $-$     &$7.171$E9 &$-$     &$2.726$\\

       method&$603$   &  $1.577$E0  & $0.991$ &$1.105$E9 &$1.456$ &$6.918$ \\

      $\mu=10^6$&$2291$   &  $4.052$E-1  & $1.018$ &$1.413$E8 &$1.541$ &$2.481$ \\

     &$8931$   &  $1.009$E-1  & $1.021$ &$1.737$E7 &$1.540$ &$6.105$\\

     &$35267$   &  $2.543$E-2  & $1.003$ &$2.179$E6 &$1.511$ &$27.60$\\
     
     &$140163$  &  $6.353$E-3  & $1.005$ &$2.735$E5 &$1.504$ &$1261$ \\
    \hline
\end{tabular*}
\end{table}

\begin{figure}[htbp]
  \centering
  \includegraphics[width=0.7\textwidth]{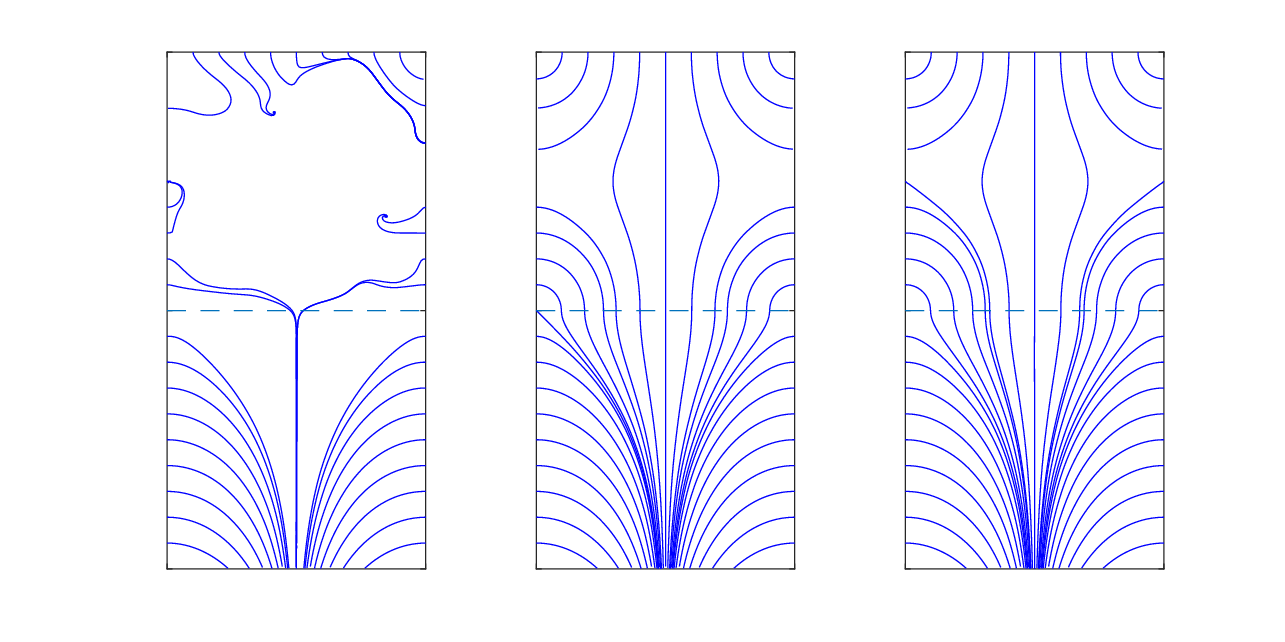}
\caption{Streamlines of velocity for classical method (left), exact solution (middle), and pressure-robust method (right) with $\mu=10^{-6}$ and $d.o.f=8931$, in Example~\ref{ex2}.}
\label{fig:streamline for ex2}
\end{figure}

\begin{figure}[htbp]
  \centering
  \includegraphics[width=0.95\textwidth]{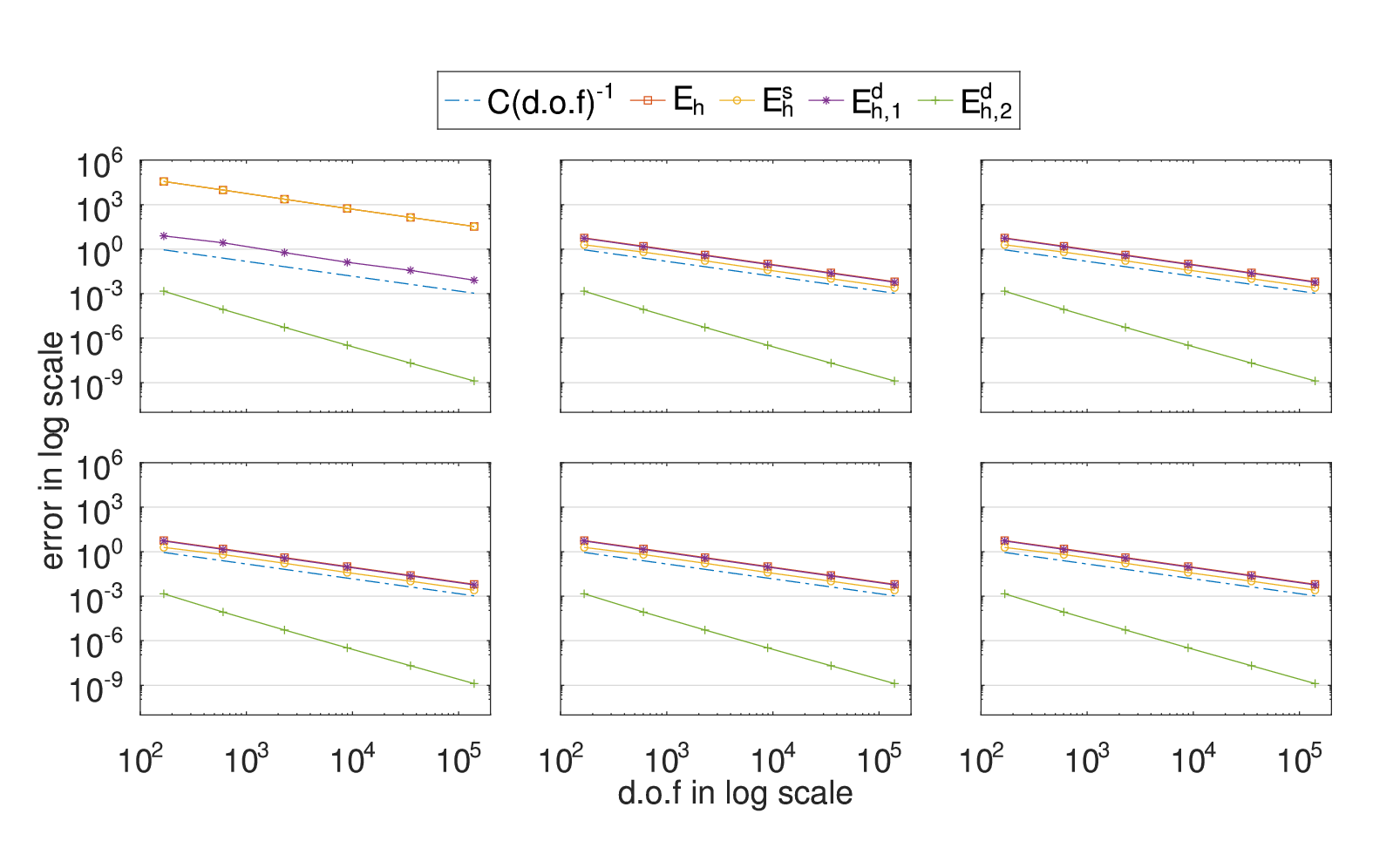}
\caption{Convergence rates of velocity for classical method (top) and pressure-robust method (bottom) with $\mu=10^{-6}$ (left), $1$ (middle), and $10^6$ (right), respectively, in Example~\ref{ex2}.}
\label{fig:convergence of velocity for ex2}
\end{figure}

\begin{figure}[htbp]
  \centering
  \includegraphics[width=0.95\textwidth]{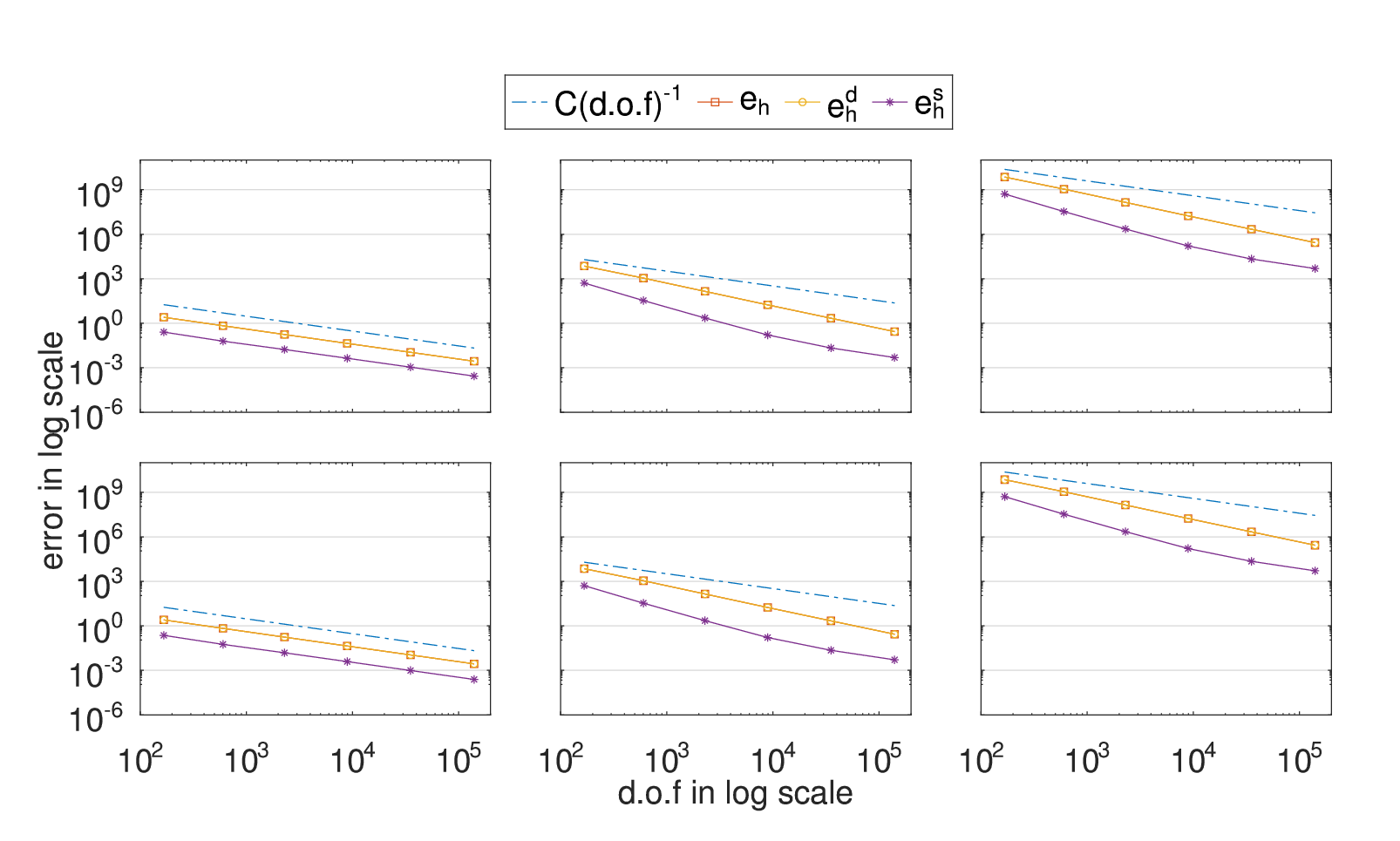}
\caption{Convergence rates of pressure for classical method (top) and pressure-robust method (bottom) with $\mu=10^{-6}$ (left), $1$ (middle), and $10^6$ (right), respectively, in Example~\ref{ex2}.}
\label{fig:convergence of pressure for ex2}
\end{figure}

\begin{figure}[htbp]
  \centering
  \includegraphics[width=0.7\textwidth]{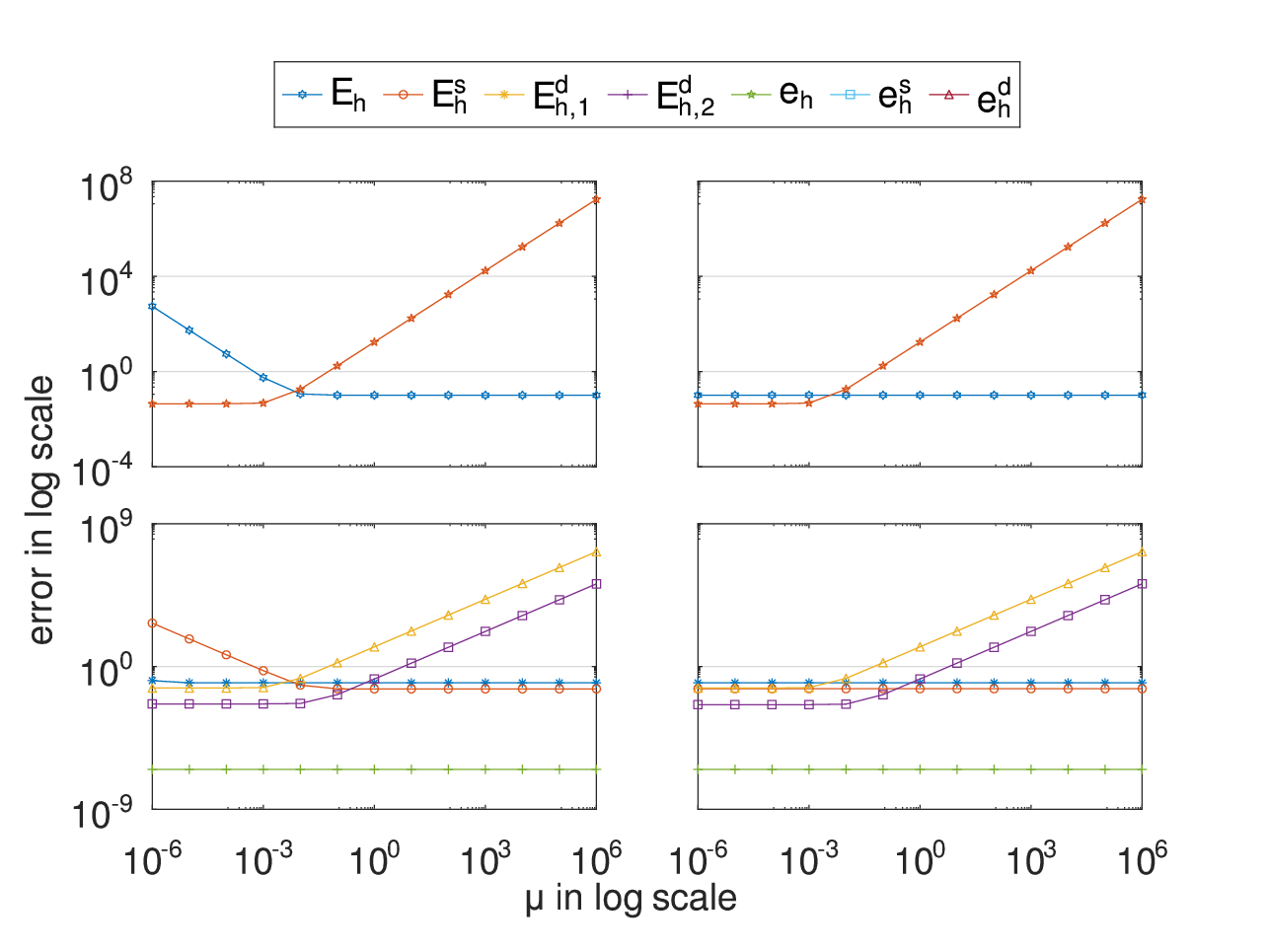}
\caption{Total errors (top) and component errors (bottom)  for classical method (left) and pressure-robust method (right) with $d.o.f=8931$ in Example~\ref{ex2}.}
\label{fig:error for ex2}
\end{figure}

\begin{figure}[htbp]
  \centering
  \includegraphics[width=0.7\textwidth]{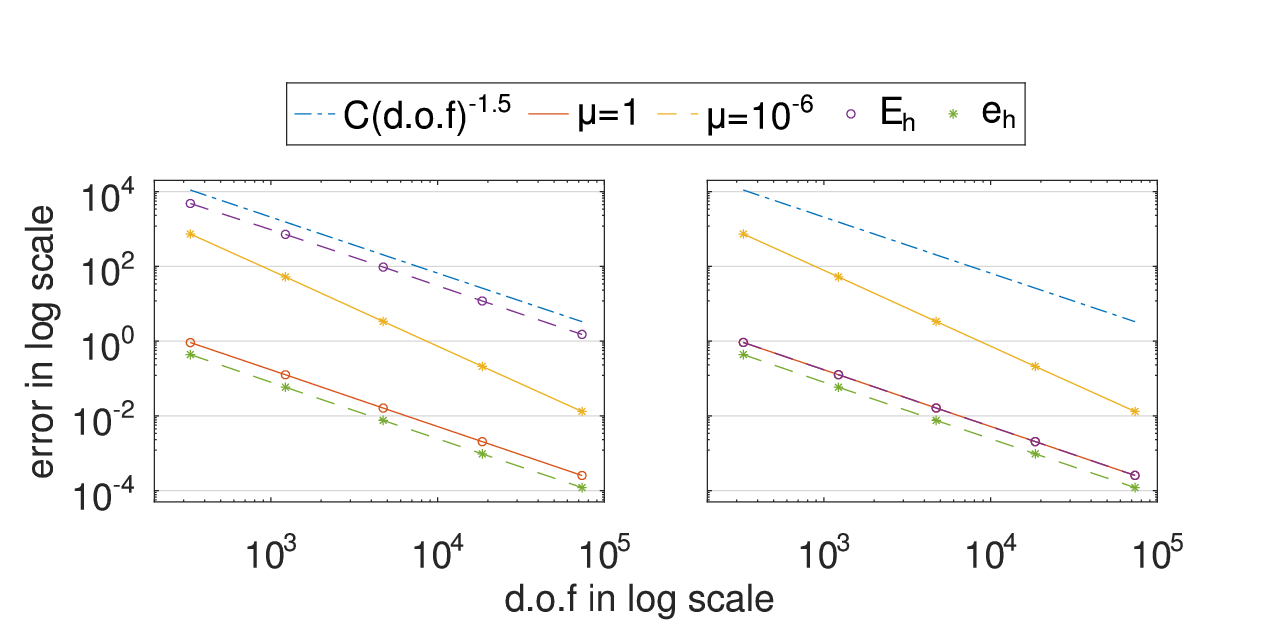}
\caption{Convergence rates for classical method (left) and pressure-robust method (right) with $k=3$ in Example~\ref{ex2}.}
\label{fig:convergence P3 for ex2}
\end{figure}

\begin{figure}[htbp]
  \centering
  \includegraphics[width=0.7\textwidth]{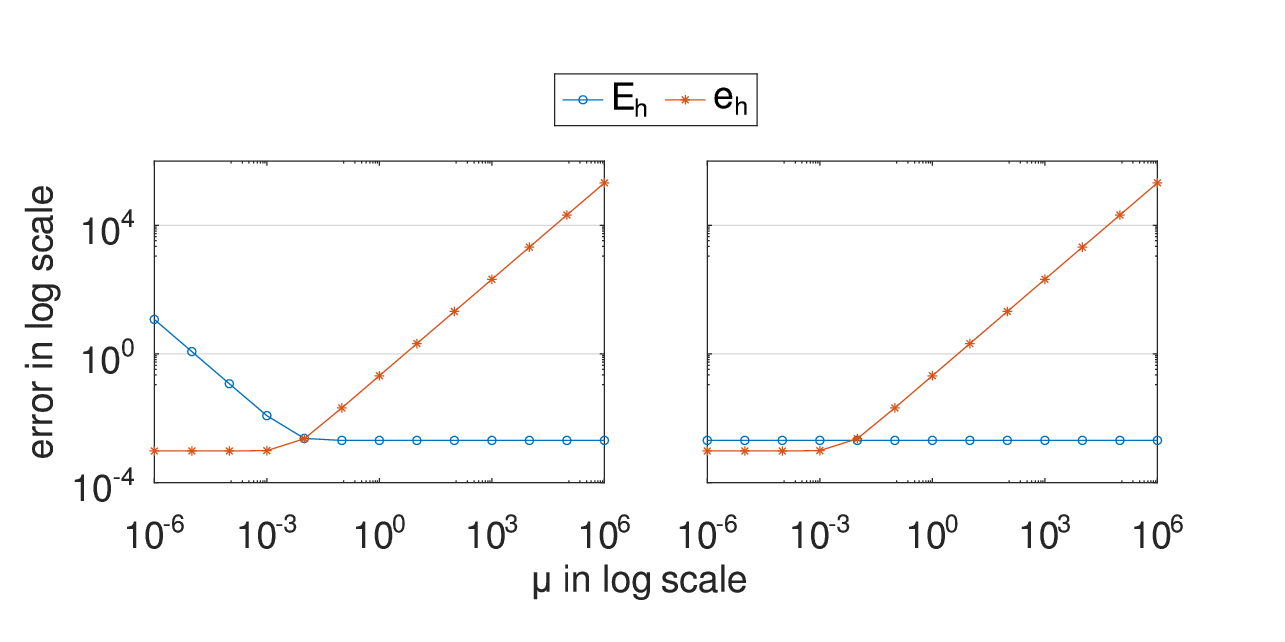}
\caption{Errors for classical method (left) and pressure-robust method (right) with $k=3$ and $d.o.f=18515$ in Example~\ref{ex2}.}
\label{fig:error P3 mu change for ex2}
\end{figure}

\subsection{Example: Lid-driven cavity with a piecewise linear interface}\label{ex3}
In this benchmark, the computational domain is the unit square $(0,1)^2$, with the Stokes region located above the Darcy region. The two subdomains are separated by a piecewise linear interface connecting the points $(0,0.5)$, $(1/3,0.42)$, $(2/3,0.58)$, and $(1,0.5)$; see Figure~\ref{fig:domian and mesh ex3} (left). An interface-fitted triangular mesh is generated using \textsc{Gmsh} and is shown in Figure~\ref{fig:domian and mesh ex3} (right), with a characteristic mesh size $h \approx 0.02$. To excite pressure effects, we add the irrotational forcing
\[
\boldsymbol{f} = \nabla\!\bigl(\lambda \sin(\pi x)\sin(\pi y)\bigr),
\]
where $\lambda$ controls its magnitude. Since this force can be absorbed into the pressure, it leaves the velocity unchanged while modifying the pressure. We set $g = 0$, the permeability $K = 10^{-4}$, and the slip coefficient $\alpha_1 = 1$, and keep these parameters fixed throughout the example. On the lid we prescribe a unit tangential velocity, whereas homogeneous Dirichlet boundary conditions are imposed on the remaining external boundaries. No analytical solution is available for this configuration, making it suitable as a benchmark to assess the practical performance and robustness of the pressure-robust formulation.

Figures~\ref{fig:classicla low and high pressure ex3} and~\ref{fig:presuure robust low and high pressure ex3} display quiver plots of the velocity field, where the color contours represent the magnitude of the velocity and the arrows indicate its direction. These plots highlight the qualitative difference between the classical and pressure-robust methods in this benchmark. For the classical scheme (Figure~\ref{fig:classicla low and high pressure ex3}), the velocity magnitude is already affected by the gradient load for $\mu = 10^{-3}$ as $\lambda$ increases from $0$ to $10^{2}$ and $10^{4}$. This effect becomes dramatic for the smaller viscosity $\mu = 10^{-6}$: the flow pattern is severely distorted and very large spurious velocities are generated, particularly along the interface and near the corners, as indicated by the change in the color scale (from order one up to several thousands). In contrast, the pressure-robust method (Figure~\ref{fig:presuure robust low and high pressure ex3}) produces velocity fields that are virtually indistinguishable for all combinations of $\mu \in \{10^{-3},10^{-6}\}$ and $\lambda \in \{0,10^{2},10^{4}\}$. Both the flow direction and the velocity magnitude remain stable and of order one. These results confirm that the classical discretization is highly sensitive to irrotational forces, with a viscosity-dependent pollution of the velocity, whereas the pressure-robust formulation effectively filters out the gradient contribution and yields a velocity field that is independent of both the pressure and the viscosity, in full agreement with the theoretical predictions.

\begin{figure}[htbp]
  \centering
  \includegraphics[width=0.4\textwidth]{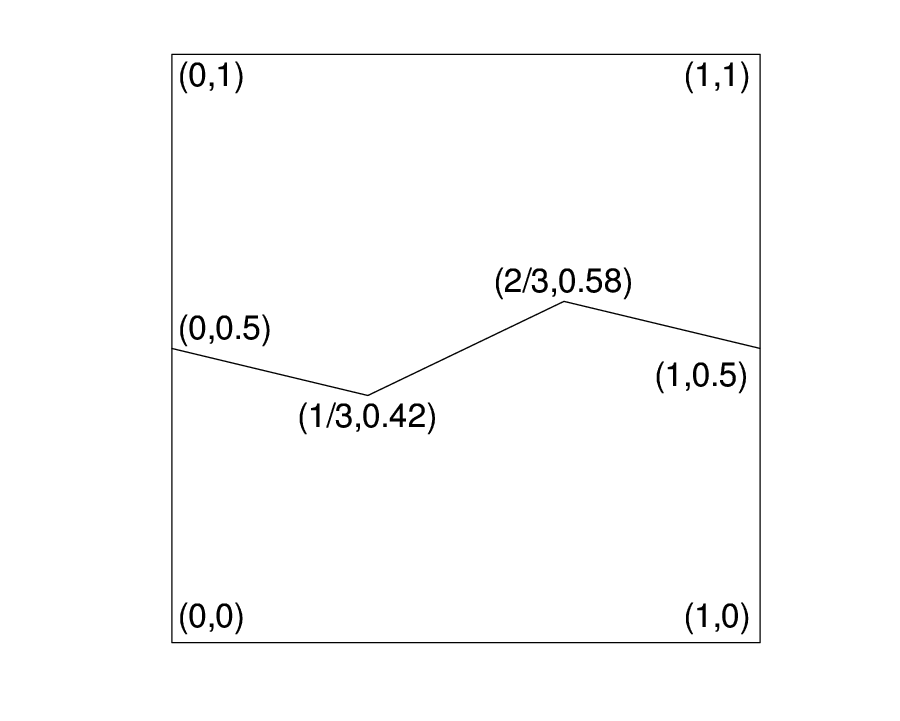}
  \includegraphics[width=0.4\textwidth]{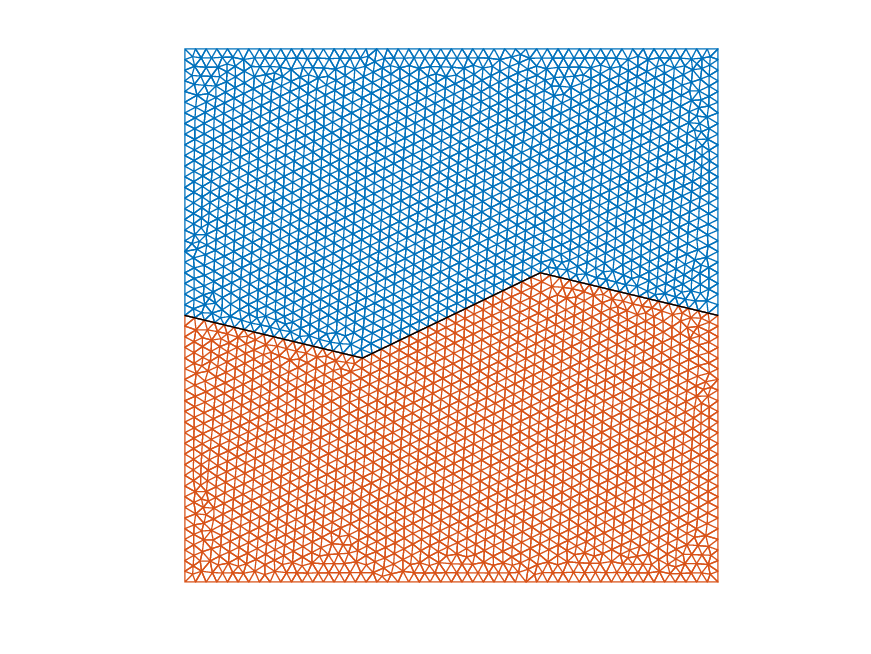}
\caption{Geometric domain (left) and mesh (right) with $h\approx0.02$ in Example~\ref{ex3}.}
\label{fig:domian and mesh ex3}
\end{figure}

\begin{figure}[htbp]
  \centering
  \includegraphics[width=0.45\textwidth]{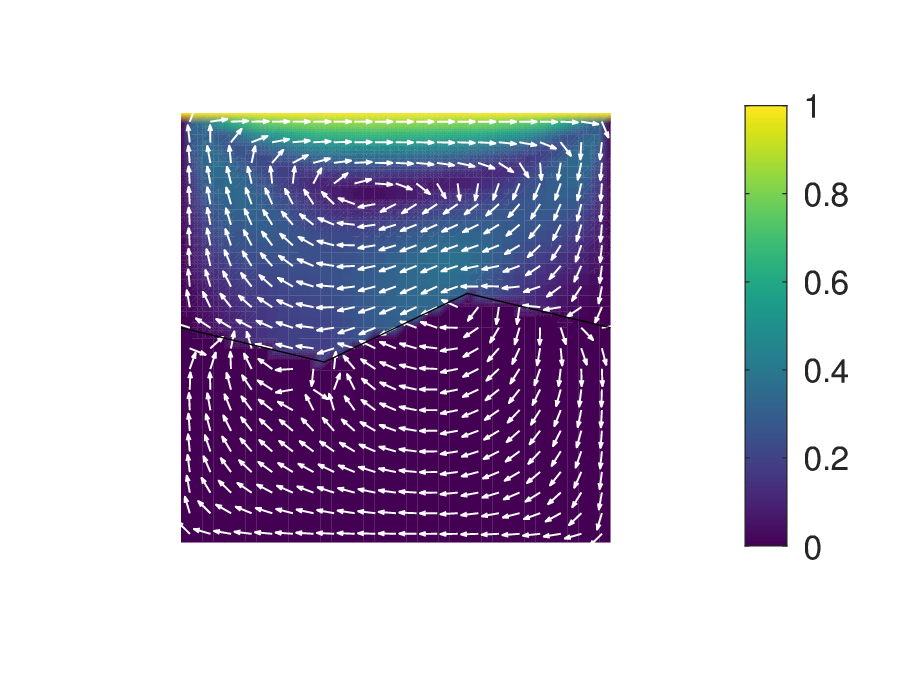}
  \includegraphics[width=0.45\textwidth]{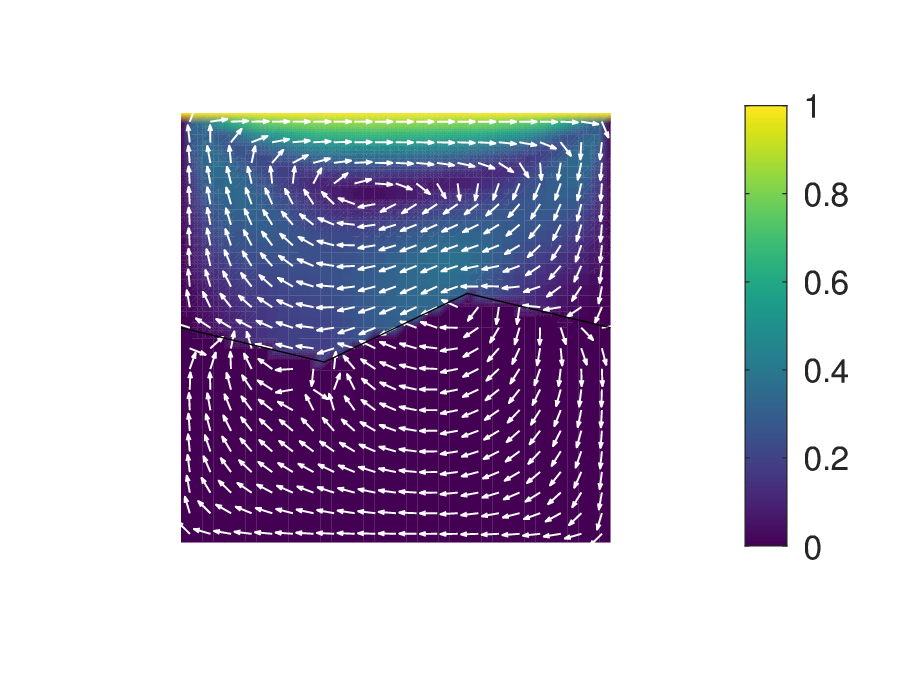}\\
  \includegraphics[width=0.45\textwidth]{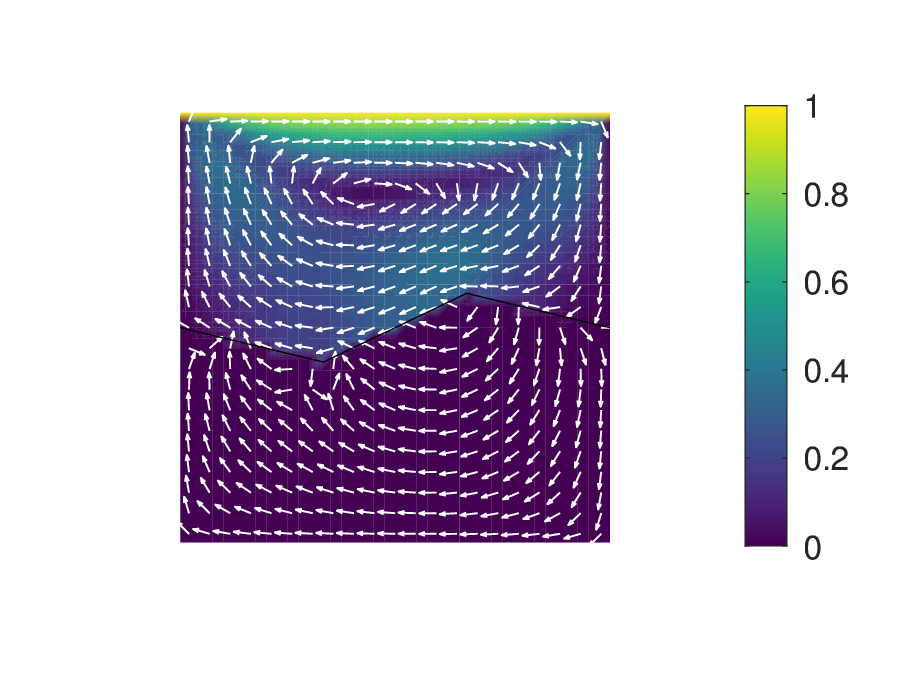}
  \includegraphics[width=0.45\textwidth]{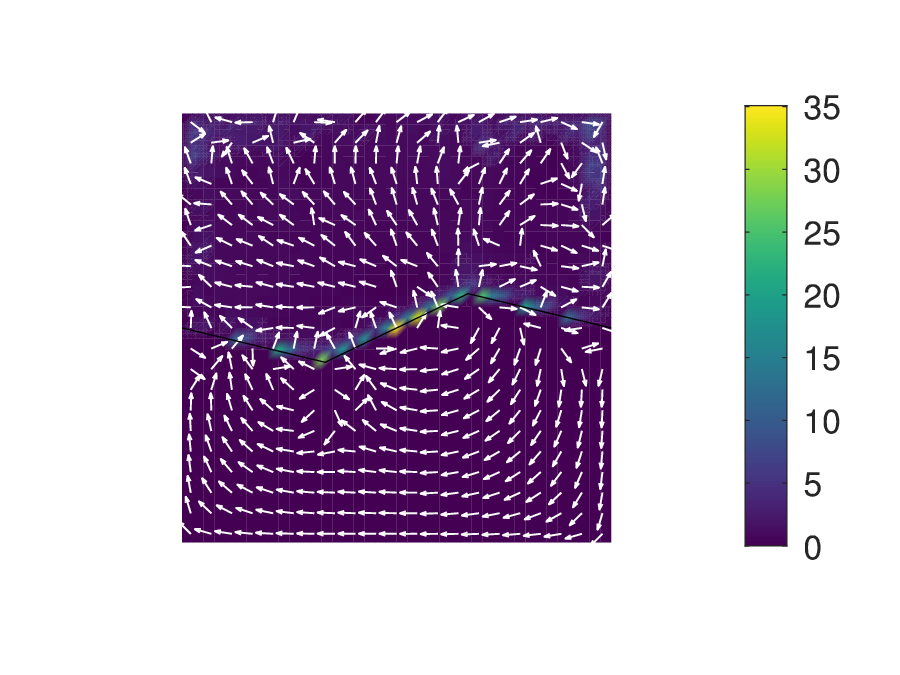}\\
  \includegraphics[width=0.45\textwidth]{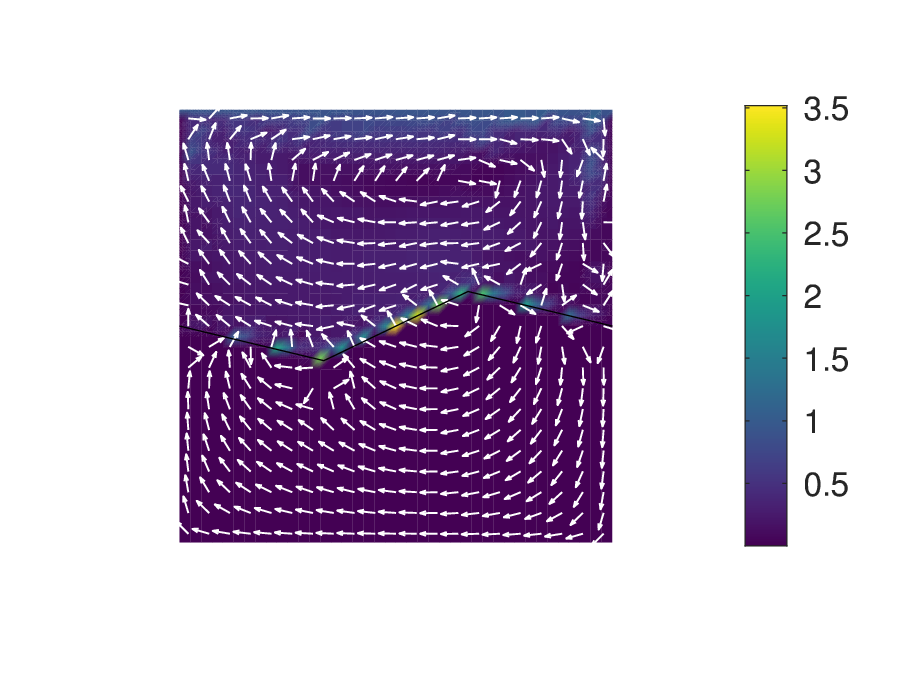}
  \includegraphics[width=0.45\textwidth]{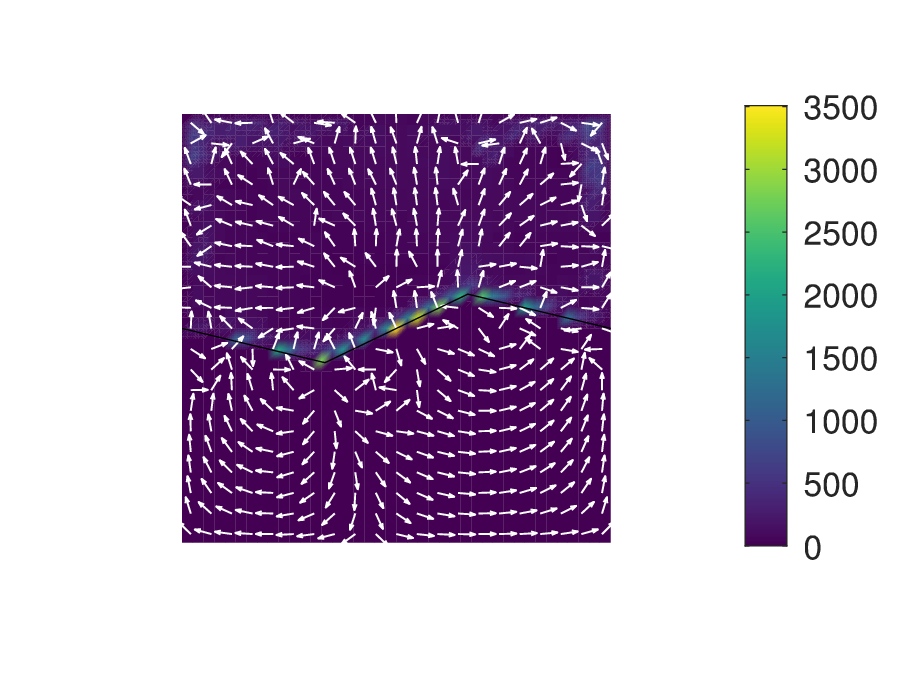}\\
\caption{Quiver plots of the classical method for $\mu=10^{-3}$ (left) and $\mu=10^{-6}$ (right), 
with $\lambda=0$ (top), $10^2$ (middle), and $10^4$ (bottom), in Example~\ref{ex3}.}
\label{fig:classicla low and high pressure ex3}
\end{figure}

\begin{figure}[htbp]
  \centering
  \includegraphics[width=0.45\textwidth]{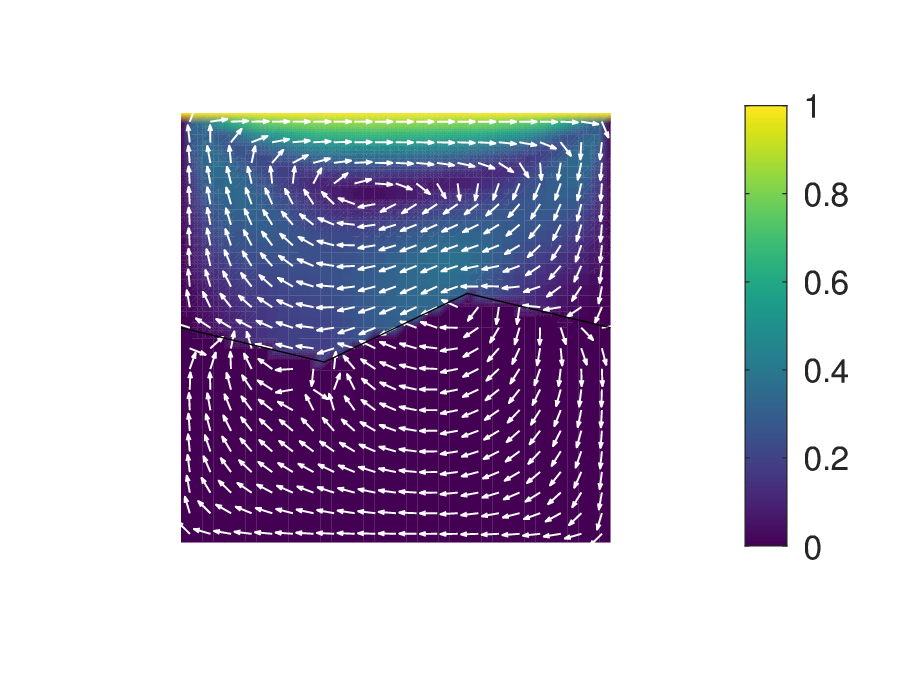}
  \includegraphics[width=0.45\textwidth]{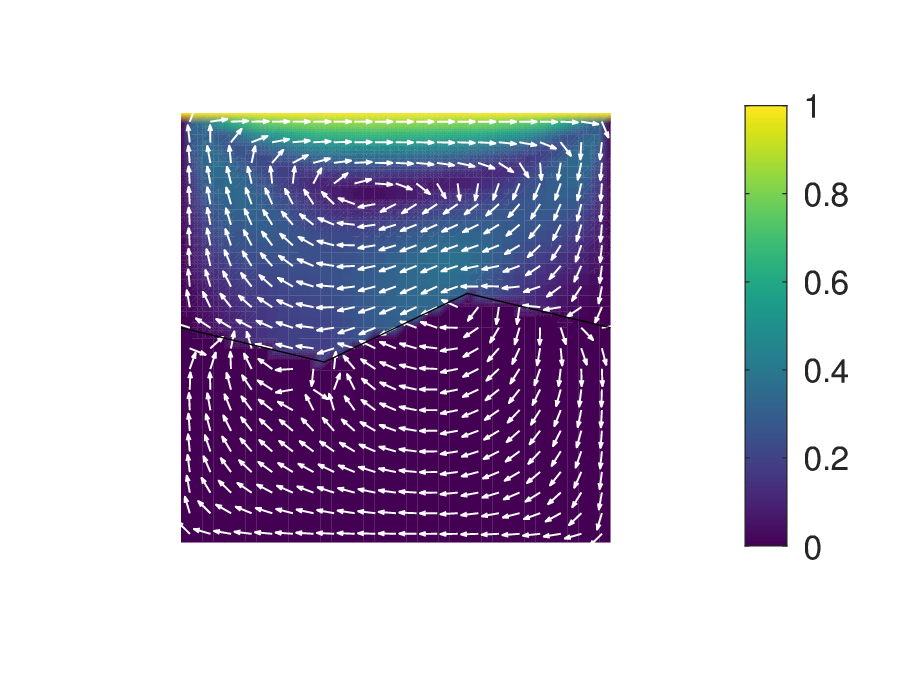}\\
  \includegraphics[width=0.45\textwidth]{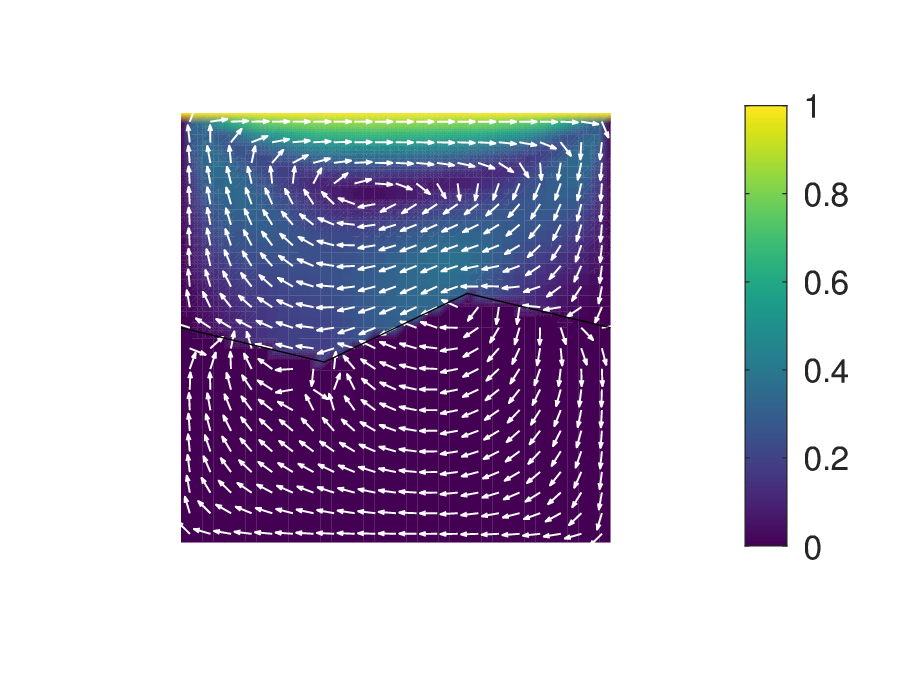}
  \includegraphics[width=0.45\textwidth]{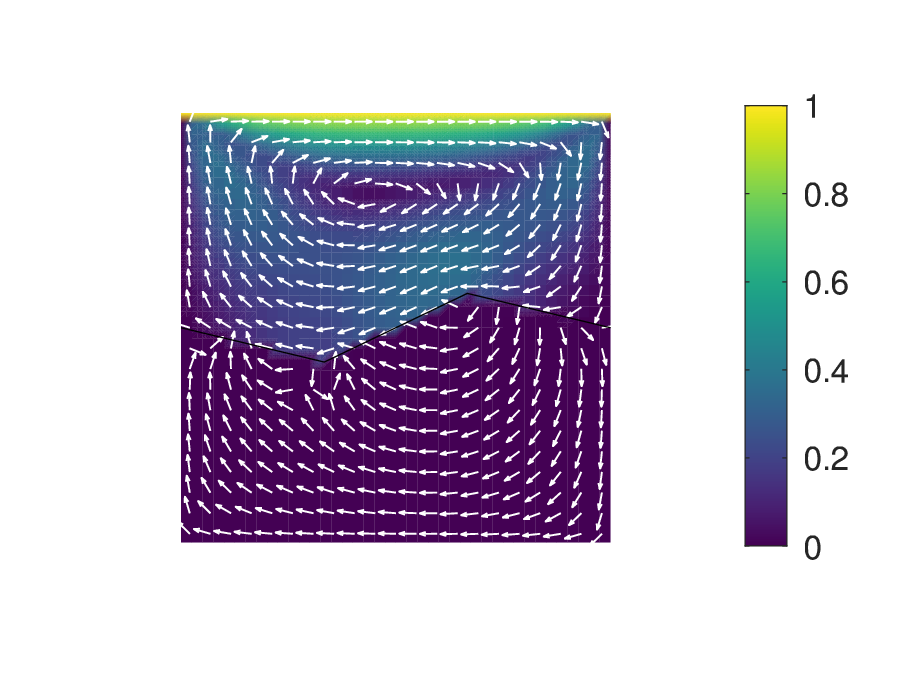}\\
  \includegraphics[width=0.45\textwidth]{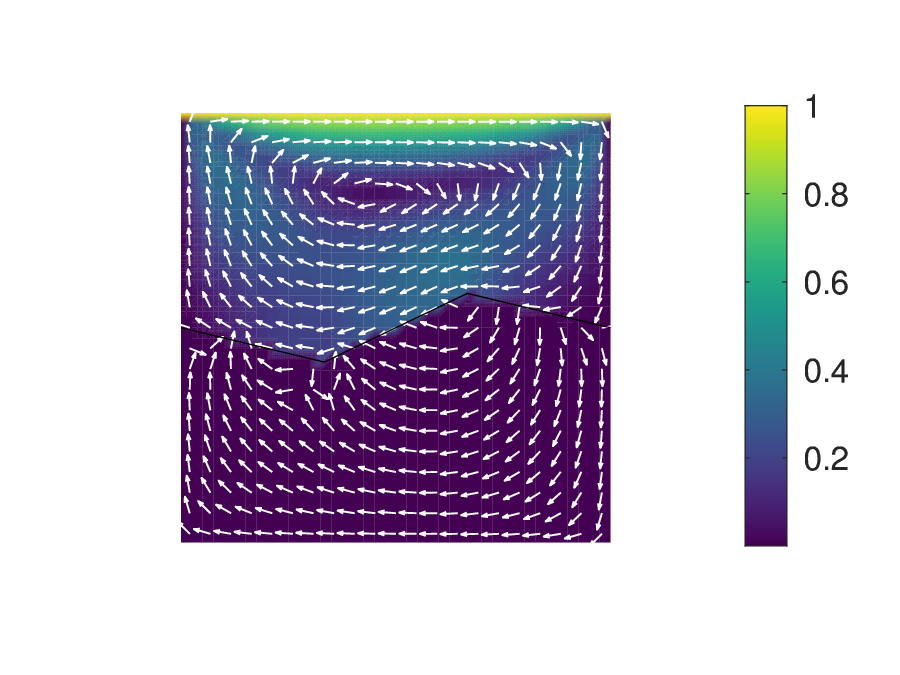}
  \includegraphics[width=0.45\textwidth]{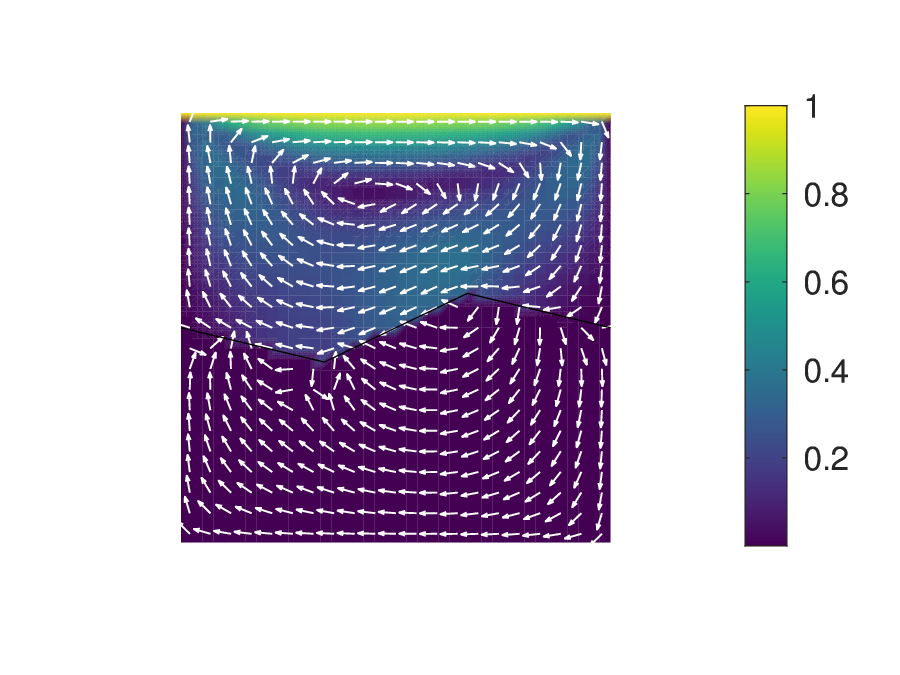}\\
\caption{Quiver plots of the pressure-robust method for $\mu=10^{-3}$ (left) and $\mu=10^{-6}$ (right), 
with $\lambda=0$ (top), $10^2$ (middle), and $10^4$ (bottom), in Example~\ref{ex3}.}
\label{fig:presuure robust low and high pressure ex3}
\end{figure}

\section{Conclusion}\label{sec:conclusion}
This study presents a comprehensive analysis of classical and pressure-robust mixed FEM for the Stokes-Darcy coupled problem, addressing critical gaps in existing methodologies. By decoupling velocity and pressure errors, we reveal how classical method incur pressure-dependent consistency errors, leading to velocity inaccuracies proportional to pressure and inversely proportional to viscosity. The proposed pressure-robust method, enhanced by divergence-free reconstruction operators, effectively eliminates these errors by enforcing exact divergence constraints and interface continuity. Numerical experiments corroborate theoretical findings, demonstrating the superiority of the pressure-robust approach in high-pressure or low-viscosity regimes. Notably, our framework overcomes the limitations of recent studies confined to low-order or two-dimensional cases, demonstrating adaptability to higher-order and three-dimensional settings. Future work could extend this approach to continuous pressure finite element spaces, time-dependent problems, or nonlinear fluid-structure interactions, thereby further enhancing the robustness and efficiency of coupled multi-physics simulations.

\section*{Acknowledgments}
The work of J. Zhang is supported by the National Natural Science Foundation of China (No. 12301469) and the Natural Science Foundation of Jiangsu Province (No. BK20210540).

\section*{Appendix: Proofs of Lemma~\ref{lem:projection operatro Upslion_1} and Lemma~\ref{lem:Pi}}
\setcounter{equation}{0}
\renewcommand\theequation{A.\arabic{equation}}
\renewcommand\thefigure{A.\arabic{figure}}
\setcounter{figure}{0}
\textbf{Lemma~\ref{lem:projection operatro Upslion_1}.}
\textit{There exists an operator $\Upsilon_h^s: V^s\rightarrow V_1^s\subset V_h^s$ satisfying, for any $T\in \mathcal{T}_h(\Omega^s)$ and all $\boldsymbol{v}^s\in V^s$,
\begin{align*}
\int_T\nabla\cdot(\Upsilon_h^s\boldsymbol{v}^s-\boldsymbol{v}^s)=0, \quad \mbox{~and~}\quad 
\|\Upsilon_h^s\boldsymbol{v}^s\|_s\lesssim \|\boldsymbol{v}^s\|_{1,s},
\end{align*}
where $V_1^s=\{\boldsymbol{v}_h^s\in V^s~|~\boldsymbol{v}_{h|T}^s\in \boldsymbol{P}_1^+(T), T\in\mathcal{T}_h(\Omega^s)\}$ with
$\boldsymbol{P}_1^+(T)=[P_1(T)]^N\oplus span\{\boldsymbol{p}_1,\cdots,\boldsymbol{p}_{N+1}\}$. 
}
\begin{proof} 
Let $\Upsilon_1^s$ be the interpolation operator from $V^s$
to
$M_1^s=\{\boldsymbol{v}_h^s\in V^s~|~\boldsymbol{v}_{h|T}^s\in [P_1(T)]^N, T\in\mathcal{T}_h(\Omega^s)\}\subset V_1^s$. For any $\boldsymbol{v}^s\in V^s$, we then have
\begin{align}\label{eqn:interpolation property}
\sum_{T\in\mathcal{T}_h(\Omega^s)}h_T^{2m-2}|\boldsymbol{v}^s-\Upsilon_1^s\boldsymbol{v}^s|_{m,T}^2\lesssim \|\boldsymbol{v}^s\|_{1,s}^2,\quad m=0,1.
\end{align}
Define the operator $\Upsilon_2^s$ from $V^s$ to $V_1^s$ by:
\begin{align}
\label{eqn:definition of Upslion_2}
\left\{
\begin{aligned}
&\Upsilon_2^s\boldsymbol{v}^s(x)=0, \quad \forall \mbox{~node~} x \mbox{~of~}\mathcal{T}_h(\Omega^s),\\
&\int_f(\Upsilon_2^s\boldsymbol{v}^s-\boldsymbol{v}^s)\cdot\boldsymbol{n}=0, \quad \forall \mbox{~side (or face)~} f \mbox{~of~}\mathcal{T}_h(\Omega^s).
\end{aligned}
\right.
\end{align}
For any $T\in\mathcal{T}_h(\Omega^s)$, it is clear
$\Upsilon_2^s\boldsymbol{v}_{|T}^s=\sum_{j=1}^{N+1}\alpha_j\boldsymbol{p}_j$ with $\alpha_j=(\int_{f_j}\boldsymbol{v}^s\cdot\boldsymbol{n}_j)/(\int_{f_j}\prod_{l=1}^N\lambda_{j_l})$,
where $f_j$ is the $j$-th side (or face) of $T$.

Let $\tilde{T}$ be a fixed reference element with $h_{\tilde{T}}=1$. For any $T\in\mathcal{T}_h(\Omega^s)$, there exists an invertible affine mapping $G$ such that $T=G(\tilde{T})$. For each $\psi: T\rightarrow \mathbb{R}$, define $\tilde{\psi}=\psi\circ G$. Analogous definitions are given for the functions defined on the sub-simplices of $T$ and $\tilde{T}$.
From the scale argument, it is clear that 
\begin{align*}
|\boldsymbol{p}_j|_{m,T}=h_T^{N/2-m}|\tilde{\boldsymbol{p}_j}|_{m,\tilde{T}},
\qquad \int_{f_j}\prod_{l=1}^N\lambda_{j_l}=h_T^{N-1}
\int_{\tilde{f}_j}\prod_{l=1}^N\tilde{\lambda}_{j_l}.
\end{align*}
And combining with trace theorem, we have
\begin{align*}
\int_{f_j}\boldsymbol{v}^s\cdot\boldsymbol{n}_j=&h_T^{N-1}
\int_{\tilde{f}_j}\tilde{\boldsymbol{v}}^s\cdot\tilde{\boldsymbol{n}}_j
\lesssim h_T^{N-1}\|\tilde{\boldsymbol{v}}^s\|_{0,\tilde{f}_j}
\lesssim h_T^{N-1}\|\tilde{\boldsymbol{v}}^s\|_{1,\tilde{T}}\\
\lesssim& h_T^{N-1}(h_T^{-N/2}\|\boldsymbol{v}^s\|_{0,T}+h_T^{1-N/2}|\boldsymbol{v}^s|_{1,T})\\
=&h_T^{N/2-1}\|\boldsymbol{v}^s\|_{0,T}+h_T^{N/2}|\boldsymbol{v}^s|_{1,T},
\end{align*}
where $|\tilde{\boldsymbol{p}}_j|_{m,\tilde{T}}$ and  $\int_{\tilde{f}_j}\prod_{l=1}^N\tilde{\lambda}_{j_l}$ are constants independent of $T$. Thus, it can be derived that
\begin{align}
\nonumber
&|\alpha_j|\lesssim  h_T^{-N/2}\|\boldsymbol{v}^s\|_{0,T}+h_T^{1-N/2}|\boldsymbol{v}^s|_{1,T},\\
\label{eqn:estimate L2 norm of Upslion_1}
&|\Upsilon_2^s\boldsymbol{v}^s|_{0,T}=|\sum_{j=1}^{N+1}\alpha_j\boldsymbol{p}_j|_{0,T}\lesssim \|\boldsymbol{v}^s\|_{0,T}+h_T|\boldsymbol{v}^s|_{1,T},\\
\label{eqn:estimate H1 norm of Upslion_1}
&|\Upsilon_2^s\boldsymbol{v}^s|_{1,T}=|\sum_{j=1}^{N+1}\alpha_j\boldsymbol{p}_j|_{1,T}\lesssim h_T^{-1}\|\boldsymbol{v}^s\|_{0,T}+|\boldsymbol{v}^s|_{1,T}.
\end{align}

From integration by parts and the definition of $\Upsilon_2^s$ in (\ref{eqn:definition of Upslion_2}), it is easy to obtain
\begin{align}
\label{eqn:div equality of Upslion_2}
\int_T \nabla\cdot(\Upsilon_2^s\boldsymbol{v}^s-\boldsymbol{v}^s)=\int_{\partial T}(\Upsilon_2^s\boldsymbol{v}^s-\boldsymbol{v}^s)\cdot\boldsymbol{n}=0.
\end{align}
Using the interpolation properties of $\Upsilon_1^s$ in (\ref{eqn:interpolation property}) and the estimates of $\Upsilon_2^s$ in (\ref{eqn:estimate L2 norm of Upslion_1}) and (\ref{eqn:estimate H1 norm of Upslion_1}), we have
\begin{align}
\label{eqn:inequality of Upslion_1}
\|\Upsilon_1^s\boldsymbol{v}^s\|_{1,\Omega^s}\leq\|\Upsilon_1^s\boldsymbol{v}^s-\boldsymbol{v}^s\|_{1,s}+\|\boldsymbol{v}^s\|_{1,s}
\lesssim \|\boldsymbol{v}^s\|_{1,s},
\end{align}
and
\begin{align}
\label{eqn:inequality of Upslion_1 and Upslion_2}
\begin{aligned}
\|\Upsilon_2^s(1-\Upsilon_1^s)\boldsymbol{v}^s\|_{1,s}^2=&\sum_{T\in\mathcal{T}_h(\Omega^s)}
\|\Upsilon_2^s(1-\Upsilon_1^s)\boldsymbol{v}^s\|_{1,T}^2\\
\lesssim &\sum_{T\in\mathcal{T}_h(\Omega^s)}h_T^{-2}\|(1-\Upsilon_1^s)\boldsymbol{v}^s\|_{0,T}^2+
|(1-\Upsilon_1^s)\boldsymbol{v}^s|_{1,T}^2\\
\lesssim & \|\boldsymbol{v}^s\|_{1,s}.
\end{aligned}
\end{align}

Setting $\Upsilon_h^s\boldsymbol{v}^s=\Upsilon_2^s(\boldsymbol{v}^s-\Upsilon_1^s\boldsymbol{v}^s)+\Upsilon_1^s\boldsymbol{v}^s$ and with the help of (\ref{eqn:div equality of Upslion_2}), (\ref{eqn:inequality of Upslion_1}), and (\ref{eqn:inequality of Upslion_1 and Upslion_2}), we can get
\begin{align*}
\int_{\Omega^s}\nabla\cdot(\Upsilon_h^s\boldsymbol{v}^s)=&\int_{\Omega^s}\nabla\cdot(\Upsilon_2^s(\boldsymbol{v}^s-\Upsilon_1^s\boldsymbol{v}^s))+\int_{\Omega^s}\nabla\cdot(\Upsilon_1^s\boldsymbol{v}^s)\\
=&\int_{\Omega^s}\nabla\cdot(\boldsymbol{v}^s-\Upsilon_1^s\boldsymbol{v}^s)+\int_{\Omega^s}\nabla\cdot(\Upsilon_1^s\boldsymbol{v}^s)\\
=&\int_{\Omega^s}\nabla\cdot\boldsymbol{v}^s,
\end{align*}
and
\begin{align*}
\|\Upsilon_h^s\boldsymbol{v}^s\|_{1,s}\leq 
\|\Upsilon_2^s(\boldsymbol{v}^s-\Upsilon_1^s\boldsymbol{v}^s)\|_{1,s}+\|\Upsilon_1^s\boldsymbol{v}^s\|_{1,s}
\lesssim \|\boldsymbol{v}^s\|_{1,s}.
\end{align*}
Thus, we complete the proof.
\end{proof}

\textbf{Lemma~\ref{lem:Pi}.}
\textit{Define $\Pi_h=\Pi_h^s\times \Pi_h^d: V^s\times V^d\rightarrow \Theta_h=\Theta_h^s\times \Theta_h^d$, which
has the following properties
\begin{align*}
&\Pi_h:V_h\rightarrow \Theta_h\cap \Theta_b,\\
&\Pi_h:V_h(0)\rightarrow \Theta_h\cap \Theta_d\cap \Theta_b.
\end{align*}
where
\begin{align*}
\Theta_b&=\{\boldsymbol{\psi}\in H(div,\Omega)~|~  
(\boldsymbol{\psi} _s\cdot\boldsymbol{n}^s)_{|\Gamma^s}=0, ~
(\boldsymbol{\psi} _d\cdot\boldsymbol{n}^d)_{|\Gamma^d}=0,  \mbox{~and~}(\boldsymbol{\psi}^s\cdot\boldsymbol{n}^s
+\boldsymbol{\psi}^d\cdot\boldsymbol{n}^d)_{|\Gamma}=0\},\\
\Theta_d&=\{\boldsymbol{\psi}\in H(div,\Omega)~|~ \nabla\cdot\boldsymbol{\psi}=0\}.
\end{align*}}
\begin{proof}
For any $\boldsymbol{v}_h\in V_h$, we have $\boldsymbol{v}_h^s=0$ on $\Gamma^s$, and $\boldsymbol{v}_h^d\cdot\boldsymbol{n}^d=0$ on $\Gamma^d$, and $\langle\boldsymbol{v}_h^s\cdot\boldsymbol{n}^s
+\boldsymbol{v}_h^d\cdot\boldsymbol{n}^d,q_h\rangle_e=0$ for any $q_h\in P_{k-1}(e), e\subset\Gamma$. From (\ref{eqn:Pi 2}) and noting that $\Pi_h\boldsymbol{v}_h\cdot\boldsymbol{n}\in P_{k-1}(e), e\subset\partial T$ for any $T\in\mathcal{T}_h$, we get $\Pi_h^s\boldsymbol{v}_h^s\cdot\boldsymbol{n}^s=0$ on $\Gamma^s$, $\Pi_h^d\boldsymbol{v}_h^d\cdot\boldsymbol{n}^d=0$ on $\Gamma^d$, and
$\Pi_h^s\boldsymbol{v}_h^s\cdot\boldsymbol{n}^s
+\Pi_h^d\boldsymbol{v}_h^d\cdot\boldsymbol{n}^d=0$ on $\Gamma$. 

Moreover, For any $\boldsymbol{v}_h\in V_h(0)$, from $\nabla\cdot(\Pi_h\boldsymbol{v}_h)_T\in P_{k-1}(T)$ and
$(\nabla\cdot(\Pi_h\boldsymbol{v}_h),1)
=(\Pi_h\boldsymbol{v}_h\cdot\boldsymbol{n},1)_{\partial\Omega}=0$, we have $\nabla\cdot(\Pi_h\boldsymbol{v}_h)\subset Q_h$. And, from the integral by part, (\ref{eqn:Pi 1}), (\ref{eqn:Pi 2}), and $b(\boldsymbol{v}_h,q_h)=0$ for any $q_h\in Q_h$, it holds
\begin{align*}
(\nabla\cdot(\Pi_h\boldsymbol{v}_h),q_h)
=&\sum_{T\in\mathcal{T}_h}(\nabla\cdot(\Pi_h^{i_T}\boldsymbol{v}_h),q_h)_T
=-\sum_{T\in\mathcal{T}_h}(\Pi_h^{i_T}\boldsymbol{v}_h,\nabla q_h)_T+\langle \Pi_h^{i_T}\boldsymbol{v}_h\cdot\boldsymbol{n},q_h\rangle_{\partial T}\\
=&-\sum_{T\in\mathcal{T}_h}(\boldsymbol{v}_h,\nabla q_h)_T+\langle \boldsymbol{v}_h\cdot\boldsymbol{n},q_h\rangle_{\partial T}
=\sum_{T\in\mathcal{T}_h}(\nabla\cdot\boldsymbol{v}_h,q_h)_T=0.
\end{align*}
According to $T\in\Omega^s$ or $\Omega^d$, $i_T$ takes values of $s$ or $d$.The above analysis means $\nabla\cdot(\Pi_h\boldsymbol{v}_h)=0$.
\end{proof}

\bibliographystyle{plain}

\end{document}